\documentclass[12pt]{amsart}

\pdfoutput=1

\usepackage{latexsym}
\usepackage[centertags]{amsmath}
\usepackage{amsfonts}
\usepackage{amssymb}
\usepackage{amsthm}
\usepackage{newlfont}
\usepackage{graphics}
\usepackage{color}

\usepackage[usenames,dvipsnames]{xcolor}

\usepackage[demo]{graphicx} 
\usepackage{wrapfig,floatrow}
\usepackage{lipsum} 
\usepackage[font=small,labelfont=bf,margin=5mm]{caption}

\RequirePackage{xparse, graphicx, caption, picins}
\DeclareDocumentCommand \addpic{O{0.4\textwidth} m g}{\parpic[r]{%
\begin{minipage}{#1}
    \includegraphics[width=\textwidth]{#2}%
    \IfNoValueTF{#3}{}{\captionof{figure}{\footnotesize #3}}
\end{minipage}
}}

\newtheorem{theo}{Theorem}[section]

\newtheorem{lem} [theo]{Lemma}
\newtheorem{cor}[theo]{Corollary}
\newtheorem{prop}[theo]{Proposition}

\theoremstyle{remark}
\newtheorem{rem}[theo]{Remark}

\newtheorem{algo}[theo]{Algorithm}

\def\SW{\texttt{SW}}
\def\NE{\texttt{EN}}

\newcommand{\area}{\operatorname{area}}
\newcommand{\coarea}{\operatorname{coarea}}
\newcommand{\dinv}{\operatorname{dinv}}

\newcommand{\bounce}{\operatorname{bounce}}
\newcommand{\cobounce}{\operatorname{cobounce}}

\newcommand{\rank}{\operatorname{rank}}

\def \LLRA {\Longleftrightarrow}

\def \TAU {{\cal T}}
\def \CD {{\cal D}}
\def \WD {\widetilde D}

\def \RA {\!\!\rightarrow\!\!}

\def\om {\omega}

\def\rank{\texttt{rank}}

\begin{document}

\vskip -.5in

\title{On the Sweep Map for Fuss Rational  Dyck Paths}

\author{Adriano M. Garsia$^1$ and Guoce Xin$^2$}

\address{ $^1$Department of Mathematics, UCSD \\
$^2$School of Mathematical Sciences, Capital Normal University,
Beijing 100048, PR China}

\email{$^1$\texttt{garsiaadriano@gmail.com}\ \  \&\small $^2$\texttt{guoce.xin@gmail.com}}

\date{May 23, 2017} 

\begin{abstract}
Our main contribution here is the discovery of a new family of standard
  Young tableaux
$
\TAU^k_n
$  which are in bijection with the family $\CD_{m,n}$
of Rational  Dyck
paths for $m=k\times n\pm1$   (the so called   ``Fuss'' case).  Using this family we
give a new proof of the invertibility of the sweep map in the Fuss case by
means of a very simple explicit algorithm. This new algorithm has running time $O(m+n)$. It
is independent of the Thomas-William algorithm.

\medskip
\noindent
\begin{small}
\emph{Keywords:} Rational Dyck paths, Sweep map, Young tableaux, $q,t$-Catalan polynomials.
\end{small}

\medskip
\noindent
\begin{small}
\emph{Math Subject Classification:} 05A19, 05E40.
\end{small}


\end{abstract}

\maketitle

\section{Introduction\label{s-introduction}}
Our focus in this paper are the so called \emph{Rational Dyck Paths} in the
$m\times n$ lattice rectangle, when $(m,n)$ is a co-prime pair of positive integers.
These paths proceed by North and East unit steps from $(0,0)$ to
$(m,n)$ remaining always  above the main diagonal (of slope $n/m$).
Figure \ref{fig:Dyck-Path-Example} illustrates an example of an $(m,n)$-Dyck path where $m=7$  and $n=5$.
\begin{figure}[!ht]
\begin{center}
   \includegraphics[height=1.5 in]{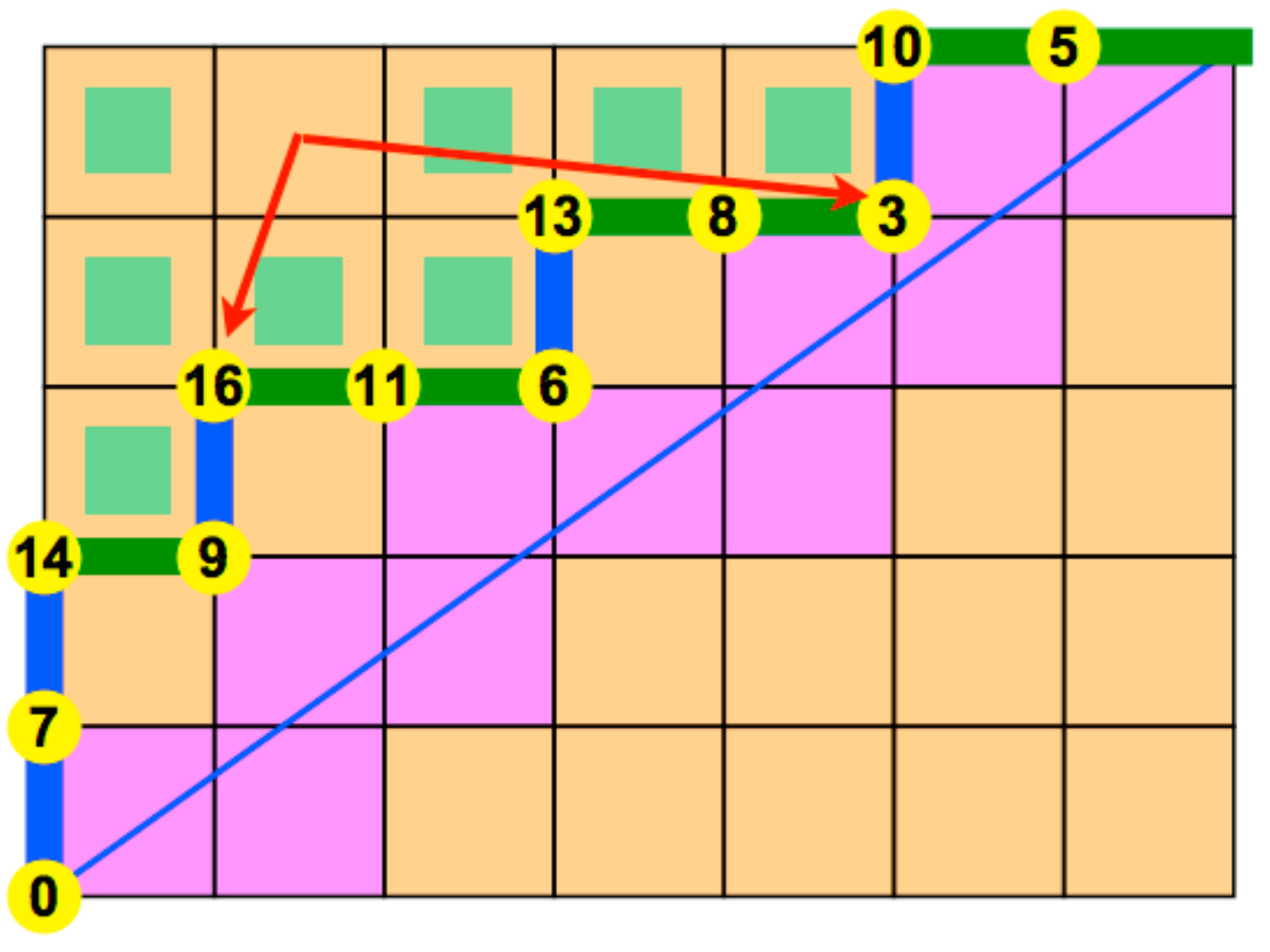}
\end{center}
\caption{An example of $(7,5)$-Dyck path.\label{fig:Dyck-Path-Example}}
\end{figure}

The coprimality of  $(m,n)$ forces the paths to remain weakly above the lattice diagonal (the set of cells cut by the main diagonal.) Each vertex of the path is assigned a \emph{rank} as follows. We start with assigning $0$ to the South
end of the first North step. This done we add an $m$ as we go  North and subtract an $n$ as we go East, as carried out in our example.

Each path   is assigned two statistics \emph{area} and  \emph{dinv}.
The area gives the number of lattice cells between the path and the lattice diagonal and  the dinv may now be simply obtained
by means of an identity proved in \cite{dinv-def} as follows. A cell of the English partition above the path contributes a unit to dinv
 if and only if the rank $a$  of the vertex on the left at the bottom of its column, and the rank  $b$
of the bottom vertex at the end of its row satisfy
 the inequalities
$
0< a-b<m+n
$. Notice that, for the path $D$ in our example, the cell
with no green square does not contribute to $dinv(D)$ since
$
0<16-3<7+5
$
is false. On the other hand,   the  cell at the top left  corner of the rectangle  does contribute to $dinv(D)$ since
$
0<14-3<7+5
$
  is true. All the cells with a green square  contribute.  Thus in this case $dinv(D)=8$. Since the lattice diagonal has $m+n-1$ cells both area
  and dinv statistics are at most $(m-1)(n-1)/2$.
For a visual definition of dinv, see \cite{dinv-area}.

 The \emph{sweep} map was conjectured to give a bijection of the family
  ${\cal D}_{m,n}$ of paths in the $m\times n$ lattice rectangle onto itself that changes $dinv$ to $area$. The construction of the sweep map is deceptively simple. Geometrically
we sweep a path $D\in \CD_{m,n}$ by letting the main diagonal of the
$m\times n$ lattice rectangle move from right to left,
and draw a North step when  we sweep the South end of a North step of $D$ and draw an East step when we sweep the West end of an East step.
The resulting path, can be shown to be in $\CD_{m,n}$ and  will be denoted here by $\Phi(D)$. The reader may find in \cite{sweepmap} all the variations, extensions and generalizations of the sweep map and who did what in this area, except the proof that it is bijective. The  proofs that it is well defined and  the $dinv$ sweeps to $area$ property can be found in \cite{dinv-area} where a visual proof and references are given. See also \cite{Gorsky-Mazin} for a bijective proof that $codinv$ sweeps to $coarea$.
 The bijectivity has been shown in a variety of special  cases including
  when $m=kn\pm1$ which is proved in \cite{Loehr-higher-qtCatalan} and \cite{Gorsky-Mazin2} and we will  refer to here as the ``Fuss'' case.
A general result proving the invertibility of a class of sweep maps that were listed in \cite{sweepmap}, was recently given by
Thomas-Williams in \cite{Nathan}. This paper is a break through in this subject after years of unsuccessful attempts at proving the invertibility.
In fact, some of the arguments in \cite{Nathan} led to the discovery in \cite{Rational-Invert} of a
simpler and purely combinatorial algorithm for inverting the sweep map for all rational Dyck paths.

The results presented here predate the  Thomas-Williams paper and our inversion algorithm, which is restricted to the Fuss case, has no connections with the Thomas-Williams algorithm. The running time of our algorithm is clearly $O(m+n)$, while previous algorithms in \cite{Nathan} and \cite{Rational-Invert} have running time  $O((m+n)\area(\Phi^{-1}(D)))$.

A more computer friendly way to construct $\Phi(D)$ is to arrange  the ranks in increasing order. For instance for the path $D$ in Figure \ref{fig:Dyck-Path-Example} this gives
the rank sequence
\begin{align}\label{e-II.1}
  r(D)=(0,3,5,6,7,8,9, 10,11, 13,14,16).
\end{align}
This done we construct a word in $S^nW^m$ (consisting of $n$ letters $S$ and $m$ letters $W$) by replacing an entry in this sequence  by an $S$ if it is the rank of the
South end of a North step and
by a $W$ if it is the rank of the West end of an East  step.
For our example
the rank sequence $r(D)$ in \eqref{e-II.1}  gives the word
\begin{align}\label{e-II.2}
\SW(\Phi(D))=SSWSSWSWWWWW.
\end{align}

We  can then use this word as a recipe for drawing $\Phi(D)$. That is we draw a South end (hence go North) when we read an $S$ and draw a West end (hence go East) when we read a $W$.  This gives the pairing of $D$  with
 a path of area $8$ as expected in Figure \ref{fig:sweep-image}.

\begin{figure}[!ht]
$$
 \hskip .20in D= \hskip -2.3in\vcenter{ \includegraphics[height= 1.2 in]{RATCAT.pdf}}
\hskip -2.3in \qquad \Longrightarrow \qquad
\Phi(D)= \hskip -2.2in \vcenter{ \includegraphics[height=1.2 in]{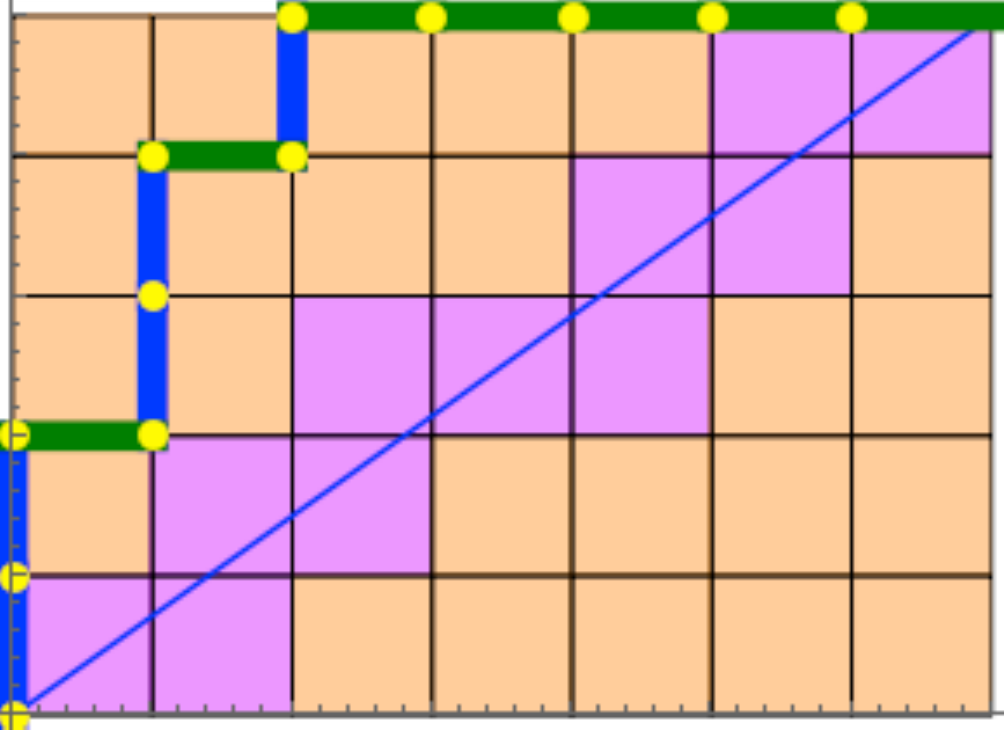}}
$$
\caption{A $(7,5)$-Dyck path and its sweep map image.\label{fig:sweep-image}}
\end{figure}

 The invertibility problem is to reconstruct $D$ from the sole knowledge of
$\SW(\Phi(D))$. Likewise,    we can construct  a word in $N^nE^m$
by using the same rank sequence. However here we replace an entry
in $r(D)$
by an $N$ if it is the rank of a North end of a North step and
by an $E$ if it is the rank of an East end of an East  step. For our example
the above sequence gives the word
\begin{align}\label{e-II.3}
\NE(\Phi(D))=EEEENEENENNN.
\end{align}

But now the situation is different, the knowledge of both $\SW(\Phi(D))$ and
$\NE(\Phi(D)) $ uniquely determines $r(D)$ and therefore $D$ itself.  From this point of view the invertibility   problem  reduces to the construction of  $\NE(\Phi(D))$ making sole use of $\SW(\Phi(D))$.
See Section \ref{s-basic}.

In the Fuss case we can use the SW word of a path $D$ to construct a standard Young tableau $T(D)$ which encodes so much information about $D$ to allow us to invert the sweep map in the simplest possible way, as we shall see.  When $m=k  n+1$
the family $\TAU_n^k$ consists of an $n\times (k+1)$ array with entries $1,2,\ldots ,m+n-1$, row and column increasing from left to right and top to bottom, with the additional property that for any pair of entries $a<d$  with $d$ directly below $a$, the entries between $a$ and $d$ form a horizontal strip. That is,  any pair of entries  $b,c$ with $a<b<c<d$ never appear in the same column.

The bijection between $\CD_{m,n}$ and $\TAU_n^k$
is constructed by the following.

\begin{algo}[Filling Algorithm]
\label{al-Filling Algorithm}

\noindent
Input: The SW-sequence $\SW(D)$ of a Dyck path $D\in \CD_{m,n}$ where $m=kn+1$.

\noindent
Output: A standard tableau $T=T(D)\in \TAU_n^k$.

\begin{enumerate}
\item   Start by placing a $1$ in the top row and the first column.

\item  If the second letter in $\SW(D)$ is an $S$ we put a $2$ on the top of the second column.

\item   If the second letter in   $\SW(D)$ is a $W$ we place $2$ below the $1$.

\item  At any  stage   the entries at the bottom of the columns but  not in row $k+1$ will be called $active$.

\item     Having placed $1,2,\cdots i-1$, we place $i$ immediately below the smallest  active entry if the $i^{th}$ letter in $\SW(D)$ is a $W$, otherwise we place $i$ at the top of the first empty column.

\item   We carry this out recursively until $1,2,\ldots ,m+n-1$ have  all  been placed.
\end{enumerate}
\end{algo}
\begin{rem}
Let $T=T(D)$ be as above. Clearly the first row of $T$ is increasing and each column is increasing. To see that $T$ is a standard Young tableau,
we observe that $T_{1,j}<T_{1,j+1}$ for all $j<n$, and then $T_{i,j}$ becomes active earlier than $T_{i,j+1}$ for all $i\le k+1$.
\end{rem}

\noindent
 This is best understood by
 an example. For  $n=3$ and $k=4$ and $D$ as in
display below
$$
{  D=\hskip -5mm \vcenter{ \includegraphics[height=.7in]{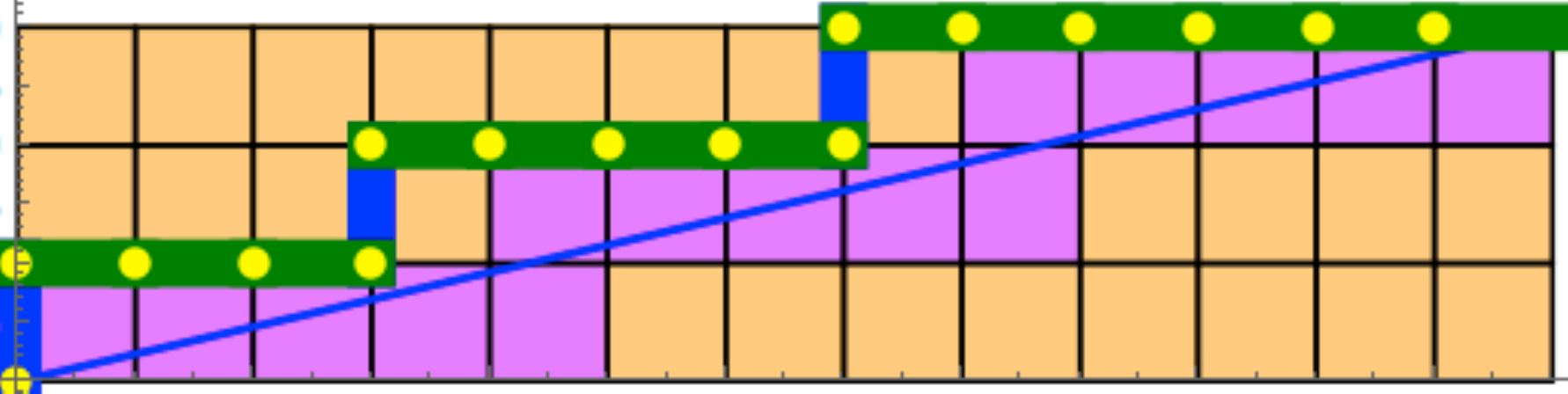} }
\atop
\hskip -3in \SW(D)=SWWWSWWWWSWWWWWW
}
\hskip -2in
{T(D)=\hskip -.3in \vcenter{\includegraphics[width=.6in]{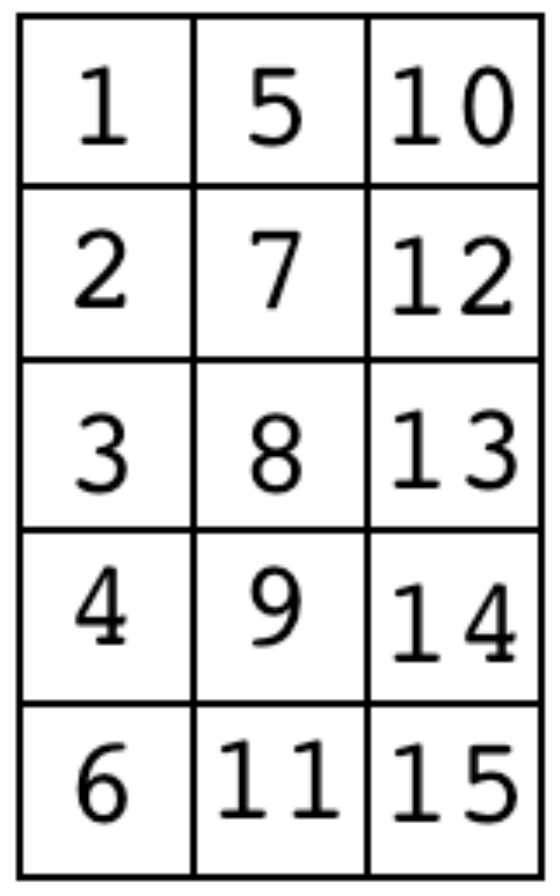} }}
$$
we obtain the tableau on the right. Notice  that $\SW(D)$ can be recovered from
$T(D)$ by  placing  letters $S$ on the positions indicated by the first row of $T(D)$ and letters $W$ in all remaining $m=kn+1$ positions.

Our  contributions  here are  the  proofs of the following   Theorems.

\begin{theo}
  \label{t-I.1}
In  the Fuss case $m=kn+1$, given $T(D)$  the word $\NE(D)$ is simply obtained by putting the letters $N$ in the positions obtained by adding $1$ to the entries in the bottom  row of $T(D)$, and putting the letters $E$
in all the remaining $m$ positions.
\end{theo}

For instance in the example above, we  get $\NE(D)$ by placing  the $N$'s in positions $7,12,16$ and the $E$'s in all the remaining $13$ positions.
We will see later how $\Phi^{-1}(D)$ itself may be easily recovered from the pair
$\SW(D)$ and $\NE(D)$.

However there is a way to  recover  $\Phi^{-1}(D)$  directly from $T(D)$
as follows.

\begin{theo}
  \label{t-I.2}
 In  the Fuss case $m=kn+1$,  the permutation $\sigma(D)$ that rearranges the letters of $\SW(D)$ in the order that gives the SW word of $\Phi^{-1}(D)$, is obtained by a walk through the entries in $T(D)$
governed by the following instructions:
\begin{enumerate}
\item Write in bold  all the entries in $T(D)$ (including $n+m$) that are by 1 more than a bottom row entry.
\item Go to 1 and write 1.
\item If you are in   row  1 go down the column to row  k+1, if the entry there is r  go to  r+1   and write  r+1.
\item If you are not in the first row go up the column one row. If the entry there  is $r$ and is not bold write $r$.
\item If the entry there  is $r$ and bold go to $r-1$ and continue until
you run into a normal entry, then write it.

\end{enumerate}
\end{theo}

Let us apply this  Algorithm to the above path $D$. In the resulting display below we have   the path $\Phi^{-1}(D)$.  On its top, we placed  the permutation produced by the algorithm and above it the resulting word $\SW\big(\Phi^{-1}(D)\big)$. For convenience, on the right is the modified tableau $T(D)$ that yielded the permutation.
$$
\hskip  .2in{
\vcenter{\includegraphics[width=3.4in]{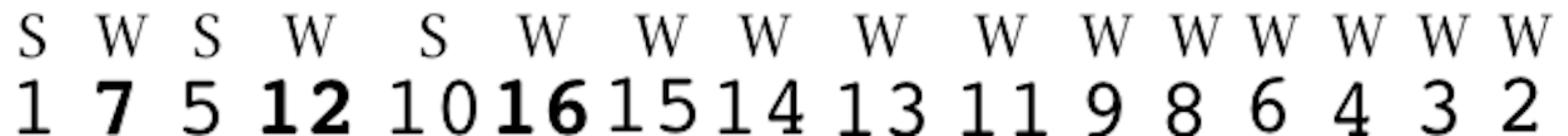} }
\atop
{\Phi^{-1}(D)=\hskip -.4in\vcenter{\includegraphics[width=3.2in]{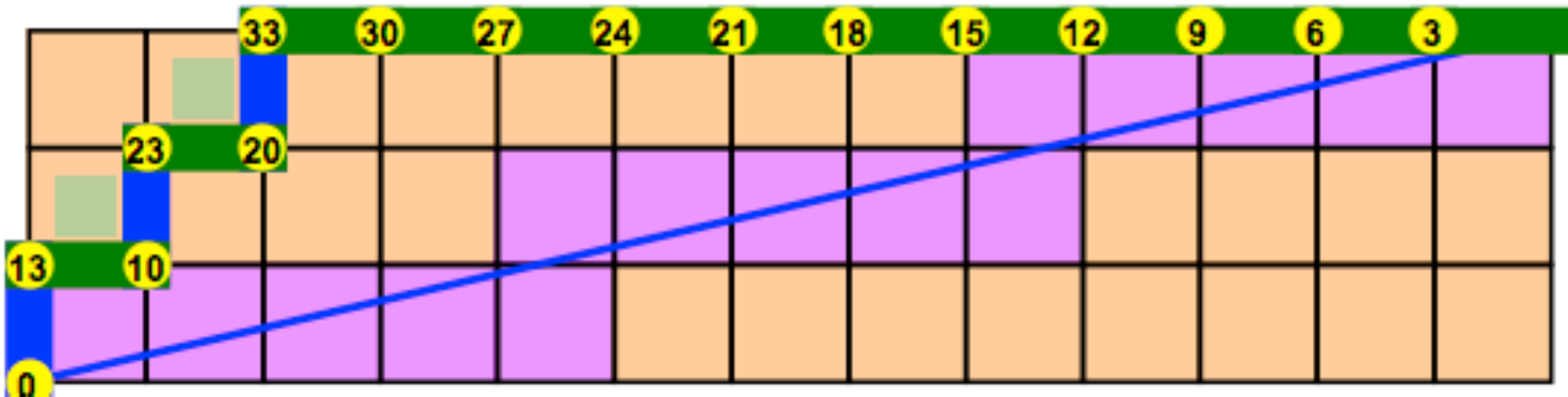} } }
}
\hskip -2.5in
{T(D)=
 \vcenter{\includegraphics[width=.6in]{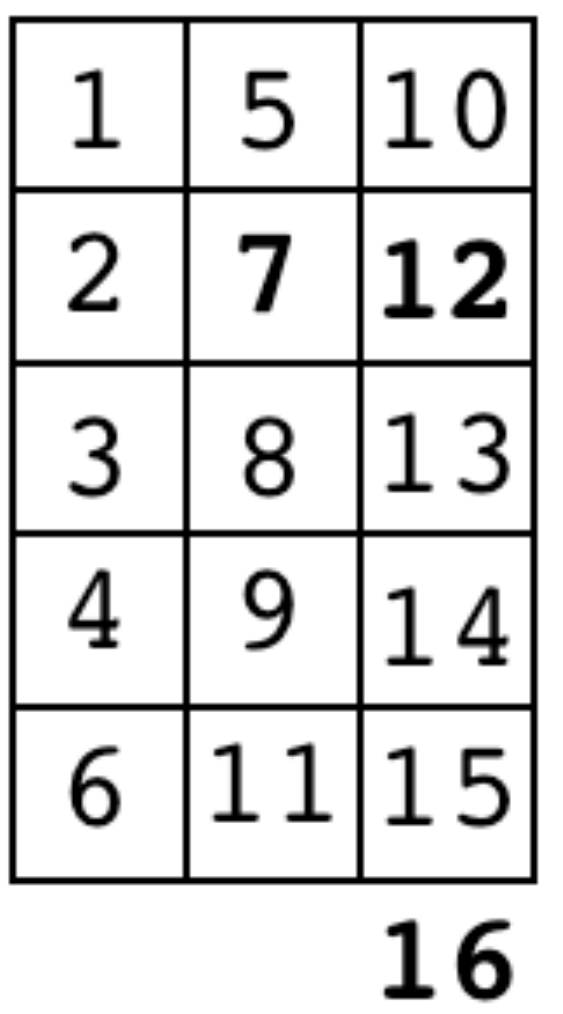} }}
$$

To obtain  $\Phi^{-1}(D)$ we proceed as follows. Having obtained the permutation we construct  the word $\SW\big(\Phi^{-1}(D)\big)$
one letter at a time by placing above each entry of the permutation an $S$ if that entry is in  the top row of $T(D)$ and a $W$ if that entry is not in the top row.
 This done we can simply draw $\Phi^{-1}(D)$ by reading the sequence of letters of $\SW\big(\Phi^{-1}(D)\big)$.

 The case $m=kn-1$ is analogous. For convenience we will separate the changes  as for $m=kn+1$.

\begin{theo}
  \label{t-I.3}
  In  the Fuss case $m=kn-1$, given $T(D)$  the word $\NE(D)$ is simply obtained by putting letters $N$ in the positions obtained by subtracting  $1$ to the entries in the bottom  row of $T(D)$, and putting letters $E$
in all the remaining $m$ positions.
\end{theo}

\begin{theo}
  \label{t-I.4}
  In  the Fuss case $m=kn-1$,  the permutation $\sigma(D)$ that rearranges the letters of $\SW(D)$ in the order that gives the successive
North and East steps of $\phi^{-1}(D)$, is obtained by a walk through the entries in $T(D)$
governed by the following instructions:
\begin{enumerate}
  \item Write in bold  all the entries in $T(D)$  that are by $1$ less than a bottom row entry.
\item Go to $1$ and write $1$.
\item If you are in   row  $1$ go down the column to row  $k+1$. If the entry there is $r$  go to  $r-1$   and write  $r-1$.
\item If you are not in the first row go up the column one row. If the entry there  is $r$ and is not bold write $r$.
\item If the entry there  is $r$ and bold go to $r+1$ and continue until
you run into a normal entry, then write it.

\end{enumerate}
\end{theo}

The paper is organized as follows. The main results are presented in this introduction. We give an explicit algorithm in Theorem \ref{t-I.2} to invert the sweep map for Fuss Dyck paths.
Section \ref{s-basic} includes the basic facts of the sweep map. We also explain the idea for proving Theorem \ref{t-I.2}. The detailed proof are presented in Section \ref{s-proof}. We discuss some combinatorial consequences in Section \ref{s-consequence}.

\section{Some basic auxiliary facts about the sweep map\label{s-basic}}

In this section we will present some observations about the sweep map and rational Dyck paths that are interesting by themselves. We will also outline the succession of steps we plan to use to prove our results. Our presentation here is aimed at conveying the basic ideas underlying    our arguments.  Proofs that are too technical will be replaced by illustrations based on examples. The corresponding formal proofs will be given in  Section \ref{s-proof}.

The proof  that $D$ can be recovered from the sole knowledge of the two words
$\SW\big(\Phi(D)\big)$  and  $\NE\big(\Phi(D)\big)$ is so elementary and simple that it must be included.
A single example should  suffice to get across all the steps of the general algorithm, called the Bipartite Algorithm.
To this end we will apply this algorithm to our first example in Figure \ref{fig:sweep-image} and show how the two words in \eqref{e-II.2} and \eqref{e-II.3} determine the rank sequence    $r(D)$ in \eqref{e-II.1}.

Let us first  label separately  the $S$ and $W$ letters  of $\SW\big(\Phi(D)\big)$ from left to right obtaining
$$
S_1S_2W_1S_3S_4W_2S_5W_3W_4W_5W_6W_7.
$$
Doing the same with $\NE\big(\Phi(D)\big)$ gives
$$
E_1E_2E_3E_4N_1E_5E_6N_2E_7N_3N_4N_5.
$$
To make sure we keep in mind  how these two words were constructed, we will
put them together as a three line array with the rank sequence in the middle row
and place on the right the path that originated them.
\begin{align}
  \begin{array}{cccccccccccc}
  S_1 & S_2 & W_1 & S_3 & S_4 & W_2 & S_5 & W_3 & W_4 & W_5 & W_6 & W_7\\
0 & 3 & 5 & 6 & 7 & 8 & 9 &  10 & 11 &  13 & 14 & 16\\
E_1 & E_2 & E_3 & E_4 & N_1 & E_5 & E_6 & N_2 & E_7 & N_3 & N_4 & N_5
\end{array}
\Leftarrow D=
\hskip -.4in  \vcenter{\includegraphics[height=1.2 in]{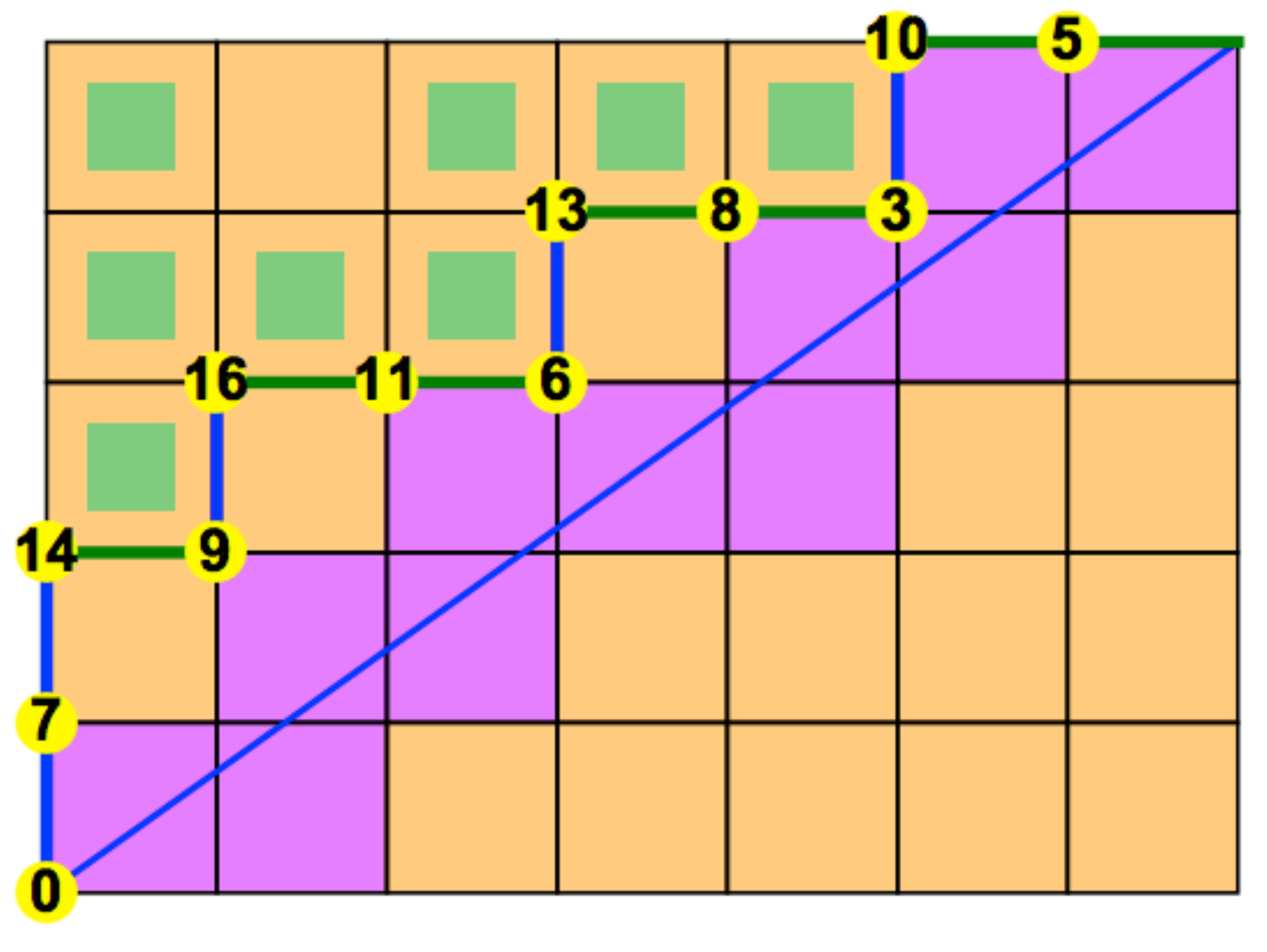}}
\label{e-1.1}
\end{align}

Recall that the rank sequence is the increasing rearrangement of the collection of ranks of the vertices  of $D$. The $SW$ word (or $SW$ sequence) $\SW\big(\Phi(D)\big)$ is obtained by  placing above a given rank $r$ an $S$ or a $W$ according to the nature of  the step of $D$
that starts at a vertex with that rank. Likewise  the $EN$ sequence
is obtained  by  placing below a given rank $r$ an $N$ or an $E$ according to the nature of  the step of $D$ that ends at a vertex with that rank. It follows that letters in the same position in both
 $\SW\big(\Phi(D)\big)$ and  $\NE\big(\Phi(D)\big)$ have the same rank. To be precise
 we should let $r(S_i)$ denote the rank of the $i$-th South end (of a North step).
 By abuse of notation, we will use the  letters themselves to denote  their  corresponding  ranks.
 In this vein we can write
 \begin{align*}
   &S_1<S_2<\cdots <S_n,\quad \  W_1<W_2<\cdots <W_m;
 \\
 &N_1<N_2<\cdots <N_n,\quad \  E_1<E_2<\cdots <E_m.
\end{align*}
Using this notation, it  follows that the ranks of the North ends
 are none other than $S_1+m,S_2+m,\ldots ,S_n+m$. Likewise
we may also identify the ranks of the East ends as
 $W_1-n,W_2-n,\ldots ,W_m-n $. Since both these sequences are increasing it follows that we may  also write
\begin{align}
  \label{e-NS-EW}
 a)\quad N_i=S_i+m\qquad  \hbox{and} \qquad
 b)\quad E_j=W_j-n .
\end{align}

This can be clearly seen in \eqref{e-1.1}. For instance, the rank of $S_3$ is $6$ so the rank of $N_3$ should be $6+7$ and accordingly above $N_3$ we have $13$. Likewise the rank of  $E_3$
is $5$ and accordingly below  $W_3$ we have $10$.

 These simple observations yield us  an algorithm for recovering the sequence of ranks directly from $\SW \big(\Phi(D)\big)$  and $\NE\big(\Phi(D)\big)$.  The idea is to construct a bipartite graph  by letting  one set of vertices  of the graph be letters of $\SW\big(\Phi(D)\big)$  and the other set of vertices be
  the letters of $\NE\big(\Phi(D)\big)$. The  edges are then    the arrows $S_i\RA N_i$, $W_j\RA E_j$ and the vertical segments joining
 letters in identical positions in $\SW \big(\Phi(D)\big)$ and $\NE\big(\Phi(D)\big)$.
 This given,  by means of the two identities in \eqref{e-NS-EW}  we will reconstruct the rank sequence $r(D)$. Again we will use the  example in \eqref{e-1.1} to communicate the general algorithm. In this case we obtain the following bipartite graph, where for simplicity we have omitted the vertical edges. The solution $r(D)$ should be understood as resulting from the  progressive construction of the unique Eulerian path that starts  and ends at the $0$ rank.
 $$
 \hskip -.2in \vcenter{\includegraphics[height=1.5 in]{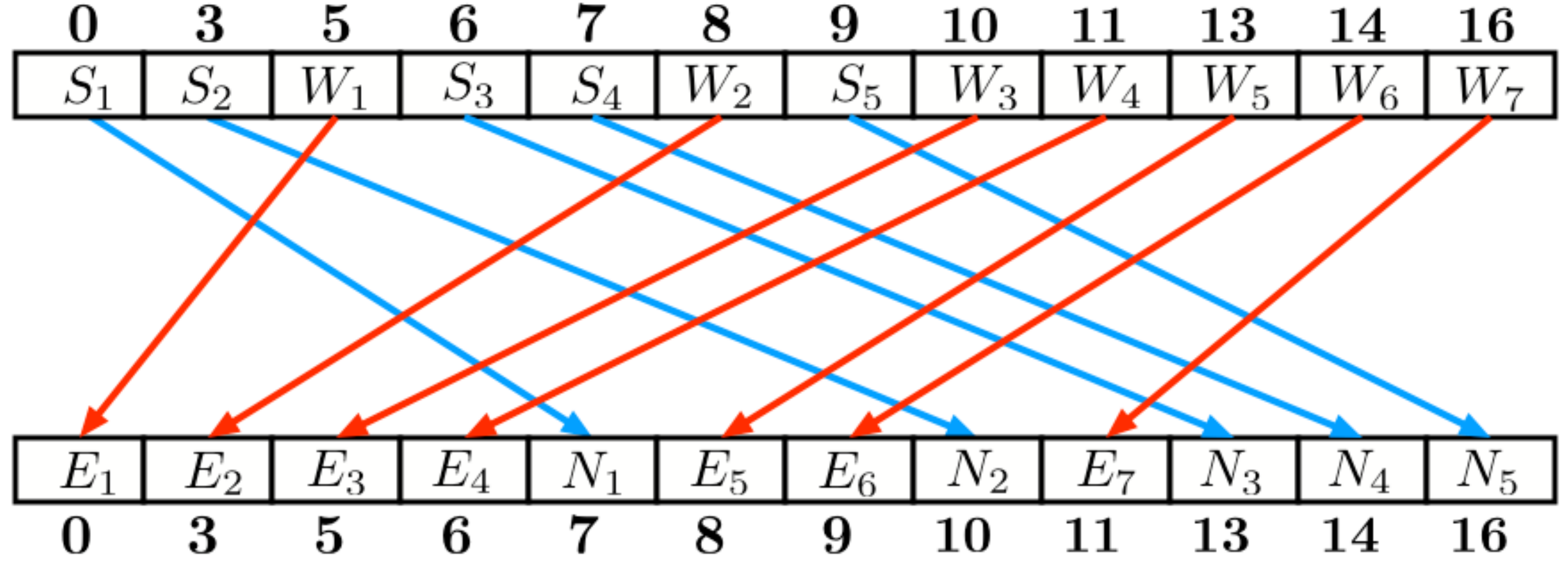}}
  \hskip -2.8in\qquad\qquad
  \Longleftarrow \hskip -.2in
  \vcenter{\includegraphics[height=1.5 in]{RATTO.pdf}}
$$

\noindent{\emph{Illustration of the Bipartite Algorithm}}

\begin{itemize}
  \item Put   $0$ above $S_1$.

\item  Follow the arrow $S_1\RA N_1$ and put   $0+7$ below $N_1$. (Using  (\ref{e-NS-EW} a)).   Put $7$ above $S_4$.

\item  Follow the arrow $S_4\RA N_4$ and put   $14$ below $N_4$, Put  $14$ above $W_6$.

\item  Follow the arrow $W_6\RA E_6$. Put   $14-5=9$ below $E_6$.  (Using  (\ref{e-NS-EW} b)). Put   $9$ above $S_5$.

\item  Follow the arrow $S_5\RA N_5$. Put   $16$ below $N_5$. Put  $16$ above $W_7$.

\item  Follow the arrow $W_7\RA E_7$. Put $11$ below $E_7$.
Put  $11$ above $W_4$.

\item  Follow the arrow $W_4\RA E_4$. Put $6$ below $E_4$.
Put  $6$ above $S_3$.

\item  Follow the arrow $S_3\RA N_3$. Put $13$ below $N_3$.
Put  $13$ above $W_5$.

\item  Follow the arrow $W_5\RA E_5$. Put $8$ below $E_5$.
Put  $8$ above $W_2$.

\item  Follow the arrow $W_2\RA E_2$. Put $3$ below $E_2$.
Put  $3$ above $S_2$.

\item  Follow the arrow $S_2\RA N_2$. Put $10$ below $N_2$.
Put  $10$ above $W_3$.

\item  Follow the arrow $W_3\RA E_3$. Put $5$ below $E_3$.
Put  $5$ above $W_1$.

\item  Follow the arrow $W_1\RA E_1$. Put $0$ below $E_1$.
Close the path.

\end{itemize}

The immediate and intended by-product of   this algorithm is the rank sequence $r(D)$.
The not intended but important byproduct is the $SW$ sequence of
$D$ itself. Indeed, recording the letters above which  we place a rank
in the above succession of steps we get the word
\begin{align}
  \label{e-1.3}
SSWSWWSWWSWW.
\end{align}
It should not be  surprising that this is indeed the $SW$ sequence of $D$,
since the  algorithm follows exactly the recipe that we used to rank   the vertices of $D$.

There is a natural involution on the bipartite graph: We rotate the graph by 180 degrees and make the exchanges $S\leftrightarrow N$ and $W \leftrightarrow E$. Then the rank sequence will become its rank complement, i.e., $M-r_{m+n}, M-r_{m+n-1}, \cdots, M-r_1$, where $M=\max(r(D))=r_{m+n}$. When we focus on the resulting Dyck paths, this involution gives the \emph{rank complement} involution
$D\mapsto \widehat{D}$ on $\CD_{m,n}$ in \cite{xin-mncore}, where $\widehat{D}$ was written as $\overline{D}$. Geometrically, if we cut $D$ at the node of the highest rank as $AB$, i.e., $A$ followed by $B$, then $\widehat{D}$ is obtained by rotating $BA$ by 180 degrees. We need to use the following result.

\begin{prop}[\cite{xin-mncore}]\label{p-rank-complement}
 Let $(m,n)$ be a coprime pair. Then the rank complement transformation preserves the dinv statistic. In other words,
for any $(m,n)$-Dyck path $D$, we have $\dinv (\widehat D)= \dinv(D).$
Consequently,
\begin{align}
  \label{e-area-hat}
  \area(\Phi(\widehat D))= \area(\Phi(D)),
\end{align}
where $\SW(\Phi(\widehat D))$ is obtained by reversing $\NE(\Phi(D))\big|_{N\to S, E\to W}$.
\end{prop}
\begin{proof}
  The first part is \cite[Corollary~15]{xin-mncore}. Identity \eqref{e-area-hat} is simply obtained by applying the sweep map and then by translating the notations.
\end{proof}

\begin{rem}
  \label{rem-1.1}
For a word  $\om  \in S^nW^m$ and $1\le i\le m+n$ denote by
$a_i(\om)$ and $b_i(\om)$,  the numbers of  ``$W$'' and ``$S$''
respectively that occur  in the first $i$ letters of $\om$.
It is important to notice that we will have $\om=\SW(D)$ for some
$D\in {\cal D}_{m,n}$  if and only if
\begin{align}
  b_i(\om)m-a_i(\om)n\ge 0 \qquad
\hbox{for all $1\le i\le m+n$}.
\label{e-1.4}
\end{align}
Indeed, we can restrict $i$ to the positions of the ``lower corners" of $\omega$ (indeed lower corners of the corresponding path), i.e., where we have $\omega_i\omega_{i+1}=WS$.

The reason is very simple. In fact after $a_i(\om)$ letters $W$ and $b_i(\om)$ letters $S$, the corresponding path has reached a lattice point of coordinates $\big(a_i(\om),b_i(\om)\big)$, this point is above the diagonal $(0,0)\RA (m,n)$ if and only if
$$
{b_i(\om)\over a_i(\om)}\ge{n\over m}.
$$
Of course the coprimality of $m,n$ forces the inequality to be strict except for $i=m+n$.
\end{rem}

When applying Remark \ref{rem-1.1} to the Fuss case, we obtain the following.
\begin{cor}\label{c-S-positions}
In the Fuss case $m=kn+1$, an increasing sequence $(t_1,t_2,\dots, t_n)$ is the positions of the $S$'s in the SW sequence of $D\in \CD_{m,n}$ if and only if $t_j \le 1+(j-1)(k+1)$ for all $j\le n$.
\end{cor}
\begin{proof}
Let $\omega\in S^nW^m$ with $S$'s in positions $t_1,\dots, t_n$. Then
 by Remark \ref{rem-1.1}, $\omega$ is the SW sequence of $D\in \CD_{m,n}$ if and only if for
$b_i(\omega)m-a_i(\omega)n\ge 0$ for all $i=t_j-1$ with $j\le n$ (these positions includes all ``lower corners" of $\omega$),
which is equivalent to
$$(j-1)m-(t_j-j)n\ge 0 \Leftrightarrow (j-1)(kn+1)-(t_j-j)n\ge 0.$$
This can be rewritten as
$$t_j-j \le (j-1)k +\frac{j-1}{n}
 \Leftrightarrow t_j\le (j-1)k+j.$$
This completes the proof.
\end{proof}

Another important consequence of Remark \ref{rem-1.1}  is the following.
\begin{lem}
  \label{l-1.1}
The Filling Algorithm terminates only when all entries $1,2,\ldots ,m+n-1$ have been placed.
\end{lem}
\begin{proof}

The only circumstance that may prematurely stop the filling of the tableau $T(D)$ is when the next letter is a $W$ and there are no remaining active entries. But that can only happen if we have completely filled the first $\l<n$ columns all the way to row $k+1$.
That means that we have processed  $\l$ letters  $S$ and $k\, \l$ letters $W$ and we are to process a $W$. But that means that the
path has reached a lattice point of rank
$$
m\, \l -k\, \l\,  n-n = \l (m-k\, n-1)+\l -n=\l-n<0,
$$
but this contradicts the inequality in \eqref{e-1.4}.
\end{proof}

There are a number of auxiliary properties of the sweep map in the Fuss case that need to be established to prove our basic results.
In particular the permutation $\sigma(D)$ produced by the algorithm stated in Theorem  \ref{t-I.2}   will be shown to consist of a single cycle which visits each of the vertices $1,2,\ldots ,n+m$ once and only once. The proof of this  property  will be divided into two parts.
We show first that our algorithm essentially defines a walk on $1,2,\ldots ,n+m$
  by following a directed graph on these vertices whose
out-degrees  and in-degrees are all equal to $1$. This will guarantee that the walk will result in a permutation. The second part of the proof will use an inductive argument   to establish that this permutation consists of a single cycle.

\def\oD{\overline{D}}

This given, using $\sigma(D)$,  let $\oD$ be the path whose   sequence  $\SW(\oD)$ is obtained by putting in position $i$ an $S$ or  $W$ according as $\sigma_i(D)$   is or is not in the first row of $T(D)$. To show that $\oD=\Phi^{-1}(D)$  we need to prove the following two properties.

\begin{enumerate}
  \item[(1)] {$\oD$ is a path in $\CD_{m,n}$.}
\item[(2)]  {the word $\SW(D)$ is obtained from the increasing rearrangement of the rank sequence  $ r(\oD )$.}
\end{enumerate}

Now recall that the components of $ r(\oD )$ are computed by reading  $\SW(\oD)$ and successively adding an $m$ when we read an $S$ or subtract an $n$ when we read a  $W$.  To make this more precise, let us denote by
 $\sigma_j(D)$  the $j^{th}$ entry in $\sigma(D)$. This given, the change in rank  caused by the $j^{th}$ step of
 $\oD$ is  $m$ if  $\sigma_j(D)$ falls in the first row of $T(D)$ and $-n$ if $\sigma_j(D)$ falls in any other  row of $T(D)$. Since we may   view the permutation $\sigma(D)$ as a walk through the tableau $T(D)$, we can recursively assign a rank to each entry of $T(D)$
 starting with $0$ rank at the first entry of the first row then follow the construction of the ranks of $\oD$
by means of $\sigma(D)$.  Thus to prove that  $D=\Phi(\oD)$  we need only show that the rank of the entry $i$ in $T(D)$ is less than the rank of the entry $i+1$. That means that $D$ itself is
obtained by rearranging the ranks of $\oD$ in increasing order. In other words,  the inverse of the permutation $\sigma(D)$ reorders the ranks of $\oD$ in increasing order.

 This observation is better understood by means of an example. Suppose that the given path $D$ (which  is Fuss with $k=2$ and $n=3$) and the corresponding tableau $T(D)$ are as depicted below.
$$
\vcenter{\includegraphics[height= 1  in]{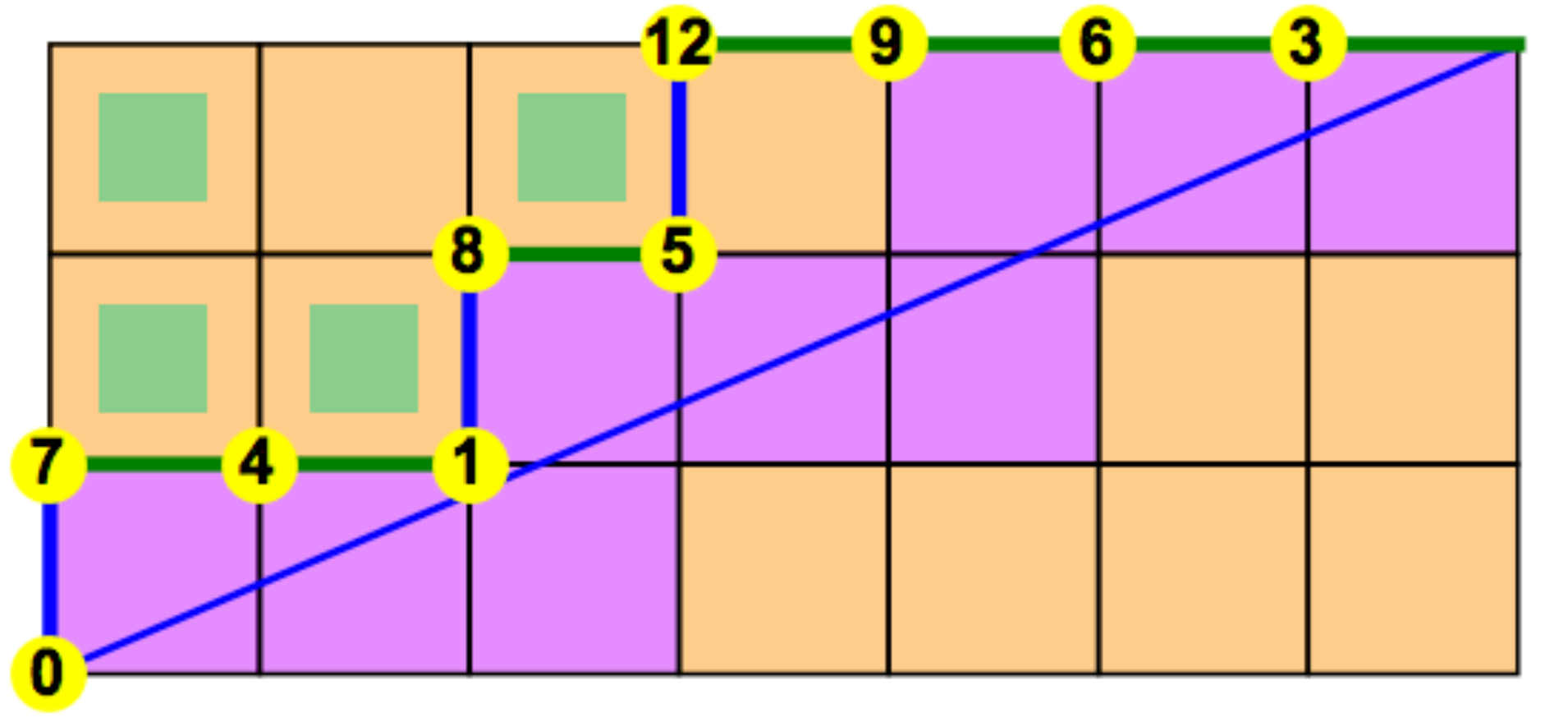}}  \hskip -3 in {\Longleftrightarrow \hskip -.28in
  \vcenter{\includegraphics[height= 1  in]{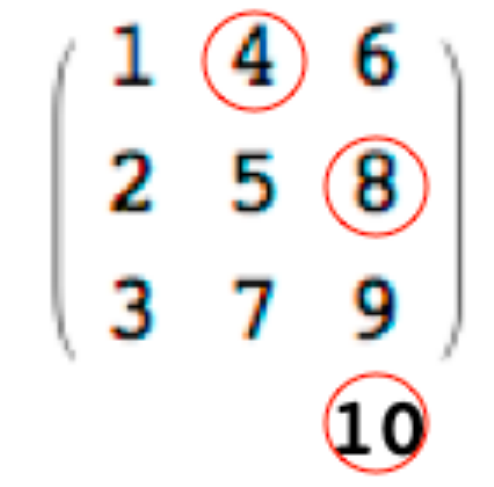}}}.
$$
Carrying out the construction of the permutation  $\sigma(D)$
and the rank sequence of $\oD$ yields the three line array and the path  $\oD$ as follows.

$$
 \vcenter{\includegraphics[height= .6  in]{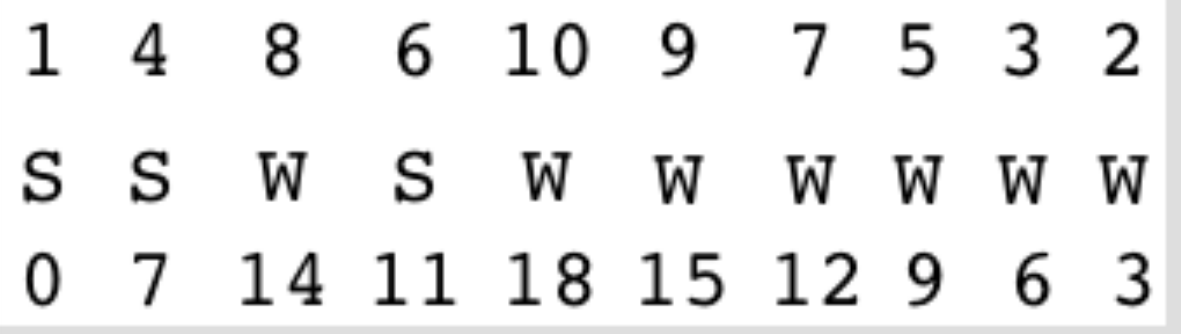}}
 \hskip -3.2 in   \Longrightarrow \hskip -.2 in {
 \vcenter{\includegraphics[height= .9  in]{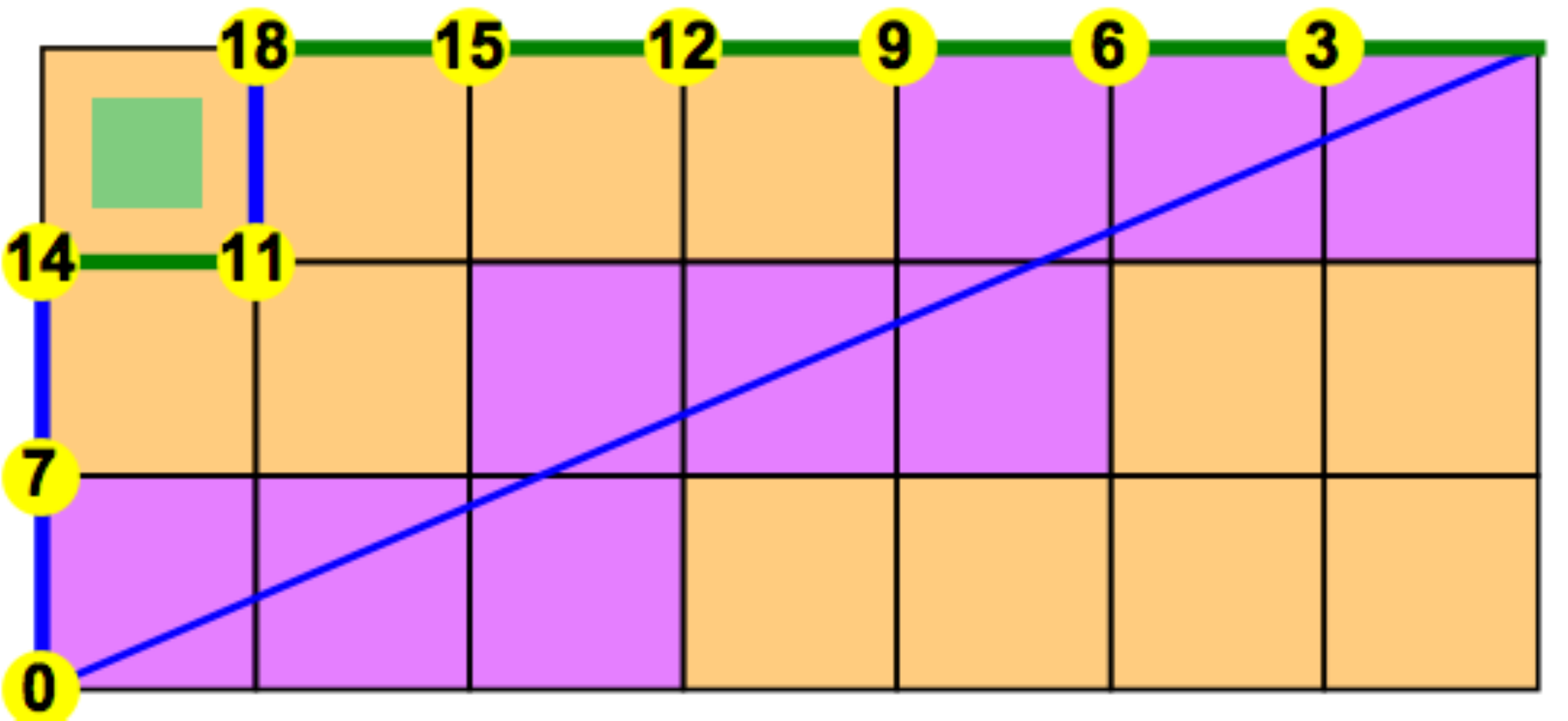}}}
$$
Here the $SW$ sequence in the middle row is obtained  by placing  a letter $S$ under each entry of $\sigma(D)$ that lies in the first row of $ T(D)$  and the letter $W$ in all other positions. This done, we use  this $SW$ sequence to draw $\oD$ and finally we obtain  the rank sequence of $\oD$ by the ranking algorithm. The third row here is constructed  by placing
under each step $S$ or $W$ the rank of its starting lattice point. Now we can visualise the punch line of our argument by simply replacing each entry in $T(D)$ by its corresponding rank. This yields  the tableau below, on the right

$$
 \vcenter{\includegraphics[height= .6  in]{SWsranks.pdf}}
 \hskip -3.2 in   \Longrightarrow \hskip -.2 in \vcenter{\includegraphics[height= 1  in]{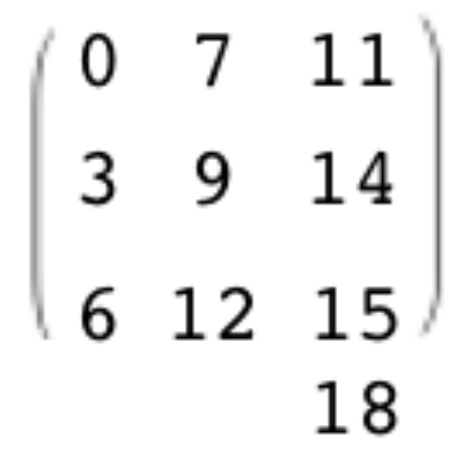}}.
$$
In this case we clearly see  that the rank at the position of $i+1$ in $T(D)$ is invariably larger than the rank at the position of $i$.
This property assures that the ranks are all positive and that reading the letters in the three line array by increasing ranks yields the SW  sequence of $D$, proving that $\oD$   is in $\CD_{kn+1,n}$ and that $\Phi(\oD)=D$.

Let us now continue with this section's presentation  of auxiliary facts.

\begin{lem}
  \label{l-1.2}For $m=kn+1$, in the directed graph that yields our walk
  through $1,2,\ldots , m+n$ all the vertices have both in-degrees and out-degrees equal to one.
\end{lem}
\begin{proof}
  The out-degree property is evident from our definition of the walk. Thus we only need to prove the in-degree property.

Let $r \in \{1,2,\ldots ,m+n\}$ be one of the entries in the
modified tableau. 
If $r$ is bold , then the in-edge is coming from an entry in row $1$ and we are done in this case. We can thus  assume that $r$ is not bold.

\begin{enumerate}

\item[i)]    If  $r$ is not in row $k+1$, then the in-edge comes  from an entry directly  below it.

\item[ii)]   If $r$ is   in row $k+1$, then $r+1$ is bold.

\begin{enumerate}
 \item[a)]  $r+1$  is not in row $k+1$. Then the in-edge is coming  from an entry directly  below $r+1$.

 \item[b)] $r+1$ is  in row $k+1$, then $r+2$ is bold
and if $r+2$ is in not in row $k+1$ then  the in-edge is coming  from an entry directly  below $r+2$. More generally assume that
$r+1,r+2,\cdots,r+u-1$ are all in row $k+1$ but $r+u$ is not in row $k+1$ then the in-edge is coming from an entry directly below $r+u$. This must eventually happen since row $k+1$ has only $n$ entries. The only exceptional case is when
$r+u=m+n$ but then $r+u$ leads to $r+u-1$ and then the in-edge comes from $m+n$.
\end{enumerate}
\end{enumerate}

This completes our proof.
\end{proof}

\begin{lem}
  \label{l-1.3}
  Let $m=kn+1$, $m'=kn'+1$ with $n'=n-1$. Then for any
$0\le a\le m'$  and $0\le b\le n'$ with $a+b<m'+m'$ we have
\begin{align}\label{e-1.5}
bm'-an'\ge 0\qquad \Longrightarrow \qquad bm -an \ge 0  .
\end{align}
\end{lem}
Before giving a formal proof, we explain a simple geometric reason.
\begin{proof}[Geometric Reasoning]
Let us call $R_{m,n}$ the $m\times n$ lattice rectangle.
 Since for $a>0$ we have the equivalence:
$$
bm'-an' > 0\qquad \LLRA \qquad \frac{b}{a}> \frac{n'}{ m'}.
$$
Then \eqref{e-1.5}  says that if the point $(a,b)\in L_{m',n'}$ is weakly  above the lattice diagonal of
$L_{m',n'}$ then the same point $(a,b)$ is weakly above the lattice  diagonal of $L_{m,n}$.
\begin{figure}
  [!ht]
$$ \vcenter{\includegraphics[width= 2.4  in]{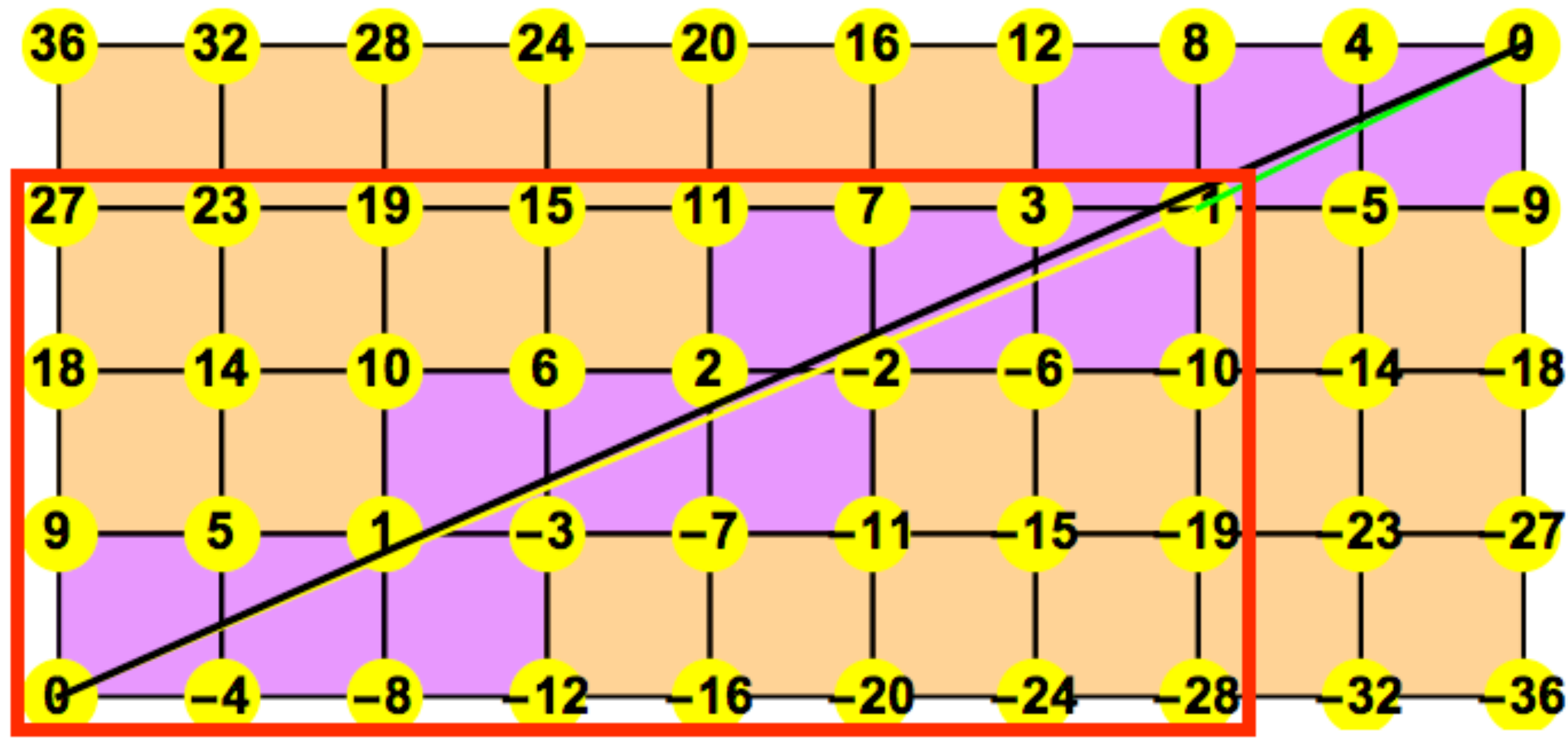}}
$$
 \caption{The lattice $L_{m,n}$ and $L_{m',n'}$.\label{fig:Lmn}}
\end{figure}

This is geometrically evident since the lattice diagonal of
$L_{m',n'}$ coincides with the portion of the lattice diagonal of   $L_{m,n}$ that is
in  $L_{m',n'}$. Figure \ref{fig:Lmn} illustrates the lattice for $k=2$ and $n=4$. For instance, the point $(a,b)$ with weight $14$ is above the diagonals for both $L_{m,n}$ and $L_{m',n'}$.
\end{proof}
\begin{proof}[Formal Proof]

We may write
$$
bm'-an' = b\big(k(n-1)+1\big)-a(n-1)=
  bm -an+a-bk.
$$
On the other hand we have
$$
bm'-an'>0 \qquad \longleftrightarrow
\qquad  bkn'+b> an'
\qquad \longleftrightarrow
\qquad   \frac{b}{  n'}> a-  bk.
$$
This gives $a-bk<1$. Thus
$$
0<bm-an+1
$$
and the coprimality of   $(m,n)$ and $(m',n')$ gives \eqref{e-1.5}.
\end{proof}

Before we can proceed further with our arguments we need to establish
the characteristic property of our standard tableaux.
This property may be stated as follows.

\begin{lem}
  \label{l-1.4}
Let $T$ be an  $n\times (k+1)$ array filled with labels $1,2,\ldots ,(k+1)n$ such that each row is increasing from left to right and each column is increasing from top to down. Then $T=T(D)$
for a Dyck path  $D\in \CD_{kn+1,n}$  if and only if
for any $a< b<c<d$, with $d$ immediately below $a$, the labels $b$ and $c$ are never
in the same column.
\end{lem}
\begin{proof}
The ``only if'' part is immediate. The fact that $d$ was placed below $a$ implies that $a$
became active as soon as it was placed and remained so until the arrival of $d$.
When $c$ arrives it cannot be inserted below $b>a$,
because that would violate  the   rule that  every new insertion is placed under the smallest active entry. Thus $b$ and $c$ cannot be in the same column.

The ``if''  part of the proof is more elaborate. We are given an $n\times (k+1)$ tableau $T$ with the stated increasing properties and the $a<b<c<d$  condition and we want to show that $T$ is obtained from the SW sequence of a  path  $D\in \CD_{kn+1,n}$ by the filling algorithm. Since the SW sequence must have letters $S$ in positions $t_1,t_2,\ldots ,t_n$, which are exactly the first row entries of $T$.
By Corollary \ref{c-S-positions}, we need to show the following.
\begin{align}
 t_{j+1}\le (k+1)j+1 \qquad (\hbox{fot all  $1\le j\le n-1$}).
\label{e-1.6}
\end{align}

To show this, notice that  in  the first $j$ columns of $T$  we have space for only
 $(k+1)j$ entries. This given, if for some $j$  we had  $t_{j+1}> (k+1)j+1$ then
 the row and column increasing condition could not allow enough space for all the entries
$a<t_{j+1}$, leading to a contradiction. Thus all the inequalities in \eqref{e-1.6} must be satisfied and the resulting path must be  a good one.

It remains to show that the entries lie in $T$ as if they were placed by our filling algorithm.
Suppose, for the sake of argument, that the entries of $T$ are placed, as we find them,  one by one in increasing order. Then again, the row and column increasing condition forces each entry
to be  placed directly under some active entry. Suppose, if possible, an entry $c$ is placed under an  active entry $b$ larger than the smallest active entry $a$, at  that  moment.
But then the entry $d$ that  is finally put under $a$ must be greater than $c$. We thus end up with
$a<b<c<d$ with $b$ and $c$ in the same column violating the $a<b<c<d$ condition.
This contradiction completes our proof of Lemma \ref{l-1.4}.
\end{proof}

As a direct consequence of Lemma \ref{l-1.4}, the Filling Algorithm is a bijection.
Lemmas \ref{l-1.3}  and \ref{l-1.4}  are needed in the  proof, by induction, that the walk consists of a single cycle. This is best illustrated by working out an example.  Here we have a
$k=3$ and $n=4$ Fuss path $D$ and its corresponding tableau $T(D)$

$$
\hskip .5in D=\hskip -2.2in\hskip 1.8in \vcenter{\includegraphics[height=1  in]{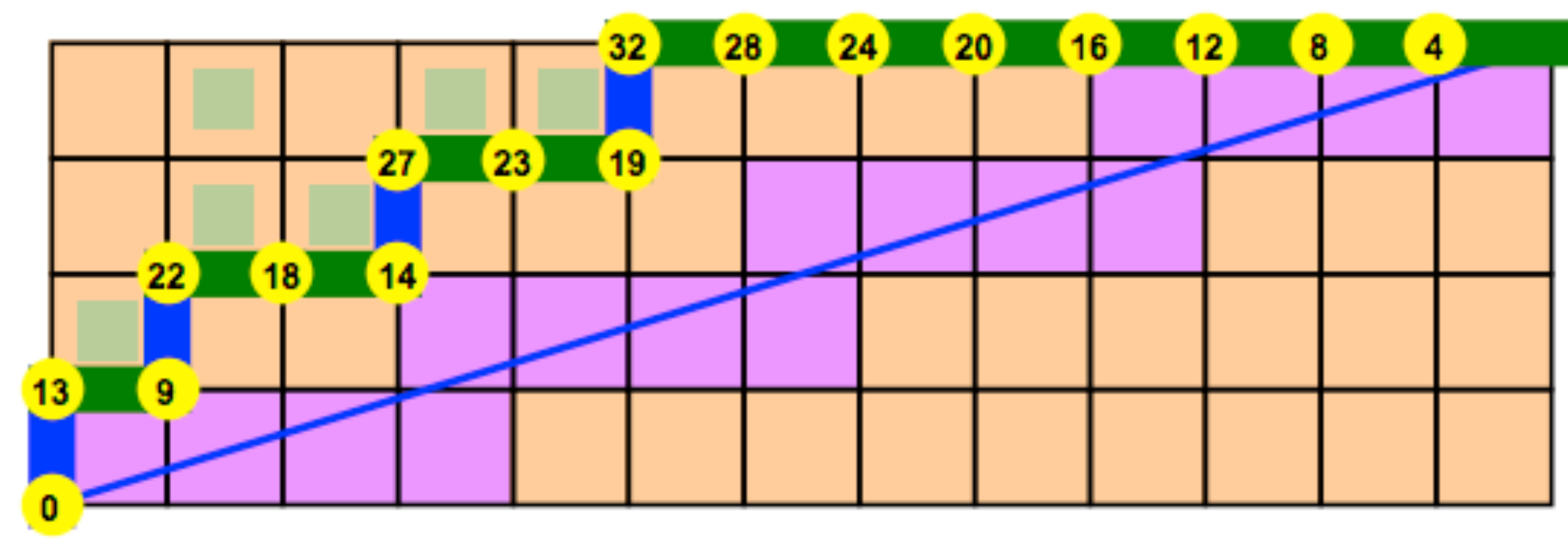}}
\hskip -3in \Longrightarrow \qquad T(D)=\begin{pmatrix}
1 & 3 & 6 &9 \cr
2 & 5 & 10 &12 \cr
4 & 8 & 13 & 14 \cr
7 & 11 & 15 &16 \cr
\end{pmatrix}.
$$
The characteristic property of these tableaux guarantees that if we remove the first column and reduce the remaining entries to a contiguous sequence,
starting with $1$,  we get the
tableau $T(D')$ of a  Fuss path $D'$. In our case $T(D)$ reduces to the
 tableau $T(D')$ displayed below
$$
T(D)=\begin{pmatrix}
1 & 3 & 6 &9 \cr
2 & 5 & 10 &12 \cr
4 & 8 & 13 & 14 \cr
7 & 11 & 15 &16 \cr
\end{pmatrix}
\qquad
\Longrightarrow
\qquad
T(D')=\begin{pmatrix}
 1 & 3 &5 \cr
2 &  6 & 8 \cr
 4 & 9 & 10 \cr
7 & 11 &12 \cr
\end{pmatrix}.
$$
The $SW$ word corresponding to $D'$ can now be constructed by placing the $S$'s in the positions indicated by the first row of $T(D')$ and $W$'s in all the remaining $13$ positions. This gives
$$
\SW(D')=SWSWSWWW WWW WWW WWW  WW
$$
and thus

$$
\hskip 1in\hskip .5in D'=\hskip -.4in\vcenter{\includegraphics[height=1.2  in]{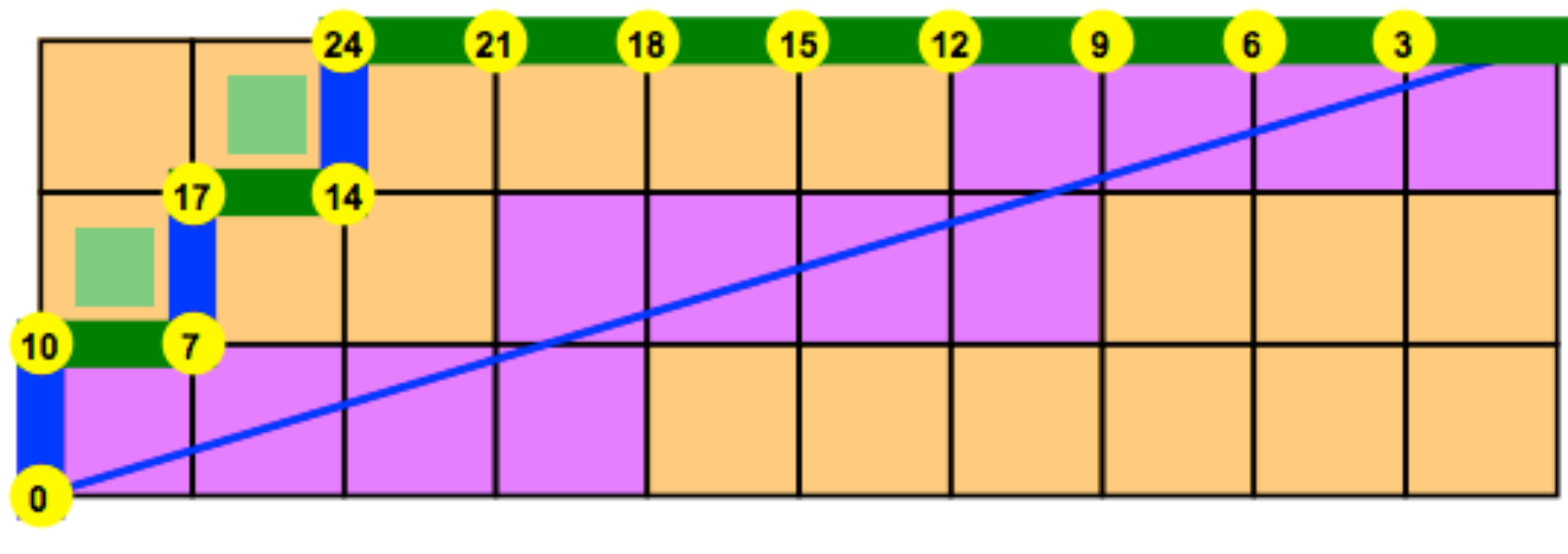}}.
$$
Applying the algorithm of Theorem \ref{t-I.2} to the tableau $T(D)$ gives the walk
\begin{align}
  w(T)=
  1 \RA 8 \RA5 \RA3 \RA12 \RA9 \RA17 \RA15 \RA13 \RA10 \RA6 \RA16 \RA14 \RA11 \RA7 \RA4 \RA 2  \RA 1.
\label{e-1.17}
\end{align}

Applying the algorithm of Theorem \ref{t-I.2} not to the tableau $T(D')$ but rather to the tableau obtained from $T(D)$ by removing the first column and treating the entries as  if they were contiguous. This  yields the tableau $T^*$
and the corresponding walk $w(T^*)$ displayed  below
\begin{align}
  \label{e-1.18}
\vcenter{{\includegraphics[width=.8  in]{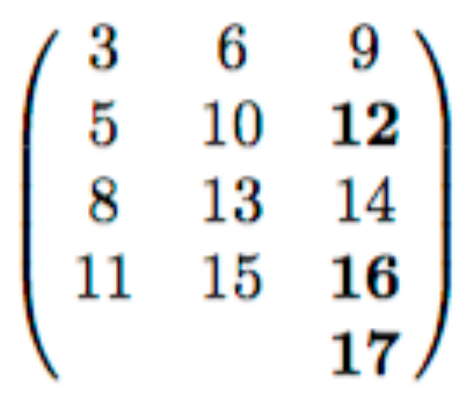}}}
\hskip -5in \Longrightarrow\quad  w(T^*)=
 3 \RA 12 \RA 9 \RA 17  \RA15 \RA 13 \RA10 \RA 6 \RA16 \RA14 \RA 11 \RA 8
 \RA 5 \RA 3.
\end{align}
 The inductive hypothesis assures that  $\sigma(T^*)$  consists of a single cycle,
as we
 clearly  see in Figure \ref{fig:single-double}. However a comparison of the
walks in \eqref{e-1.18} and \eqref{e-1.17} reveals  that the  cycle  given by \eqref{e-1.18} has an extension
that  reveals the cyclic nature of the walk in \eqref{e-1.17}. \begin{figure}[!ht]
$$
 {\qquad\ \ \ \includegraphics[width=1.3  in]{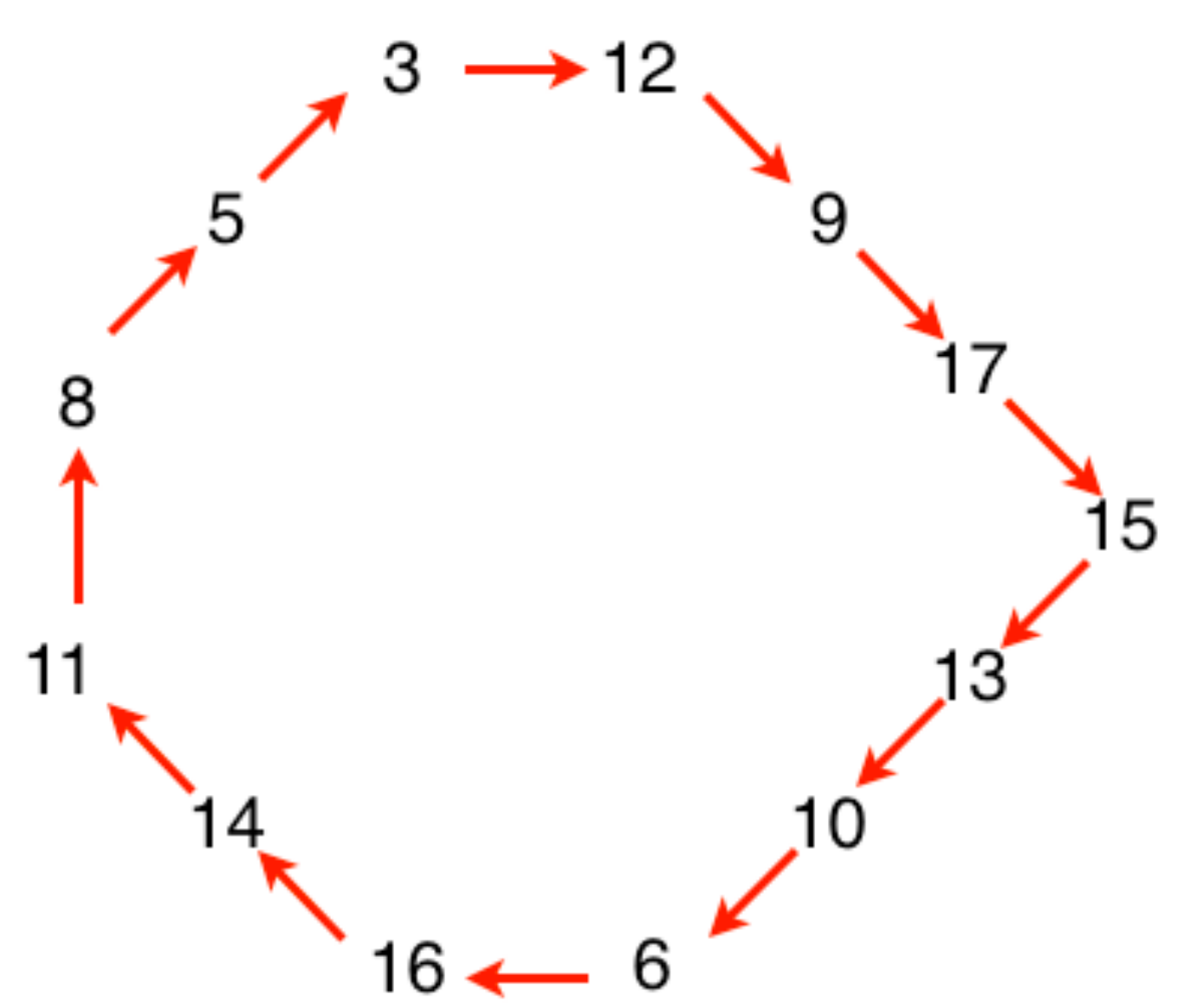}
 \atop \includegraphics[width=1.7  in]{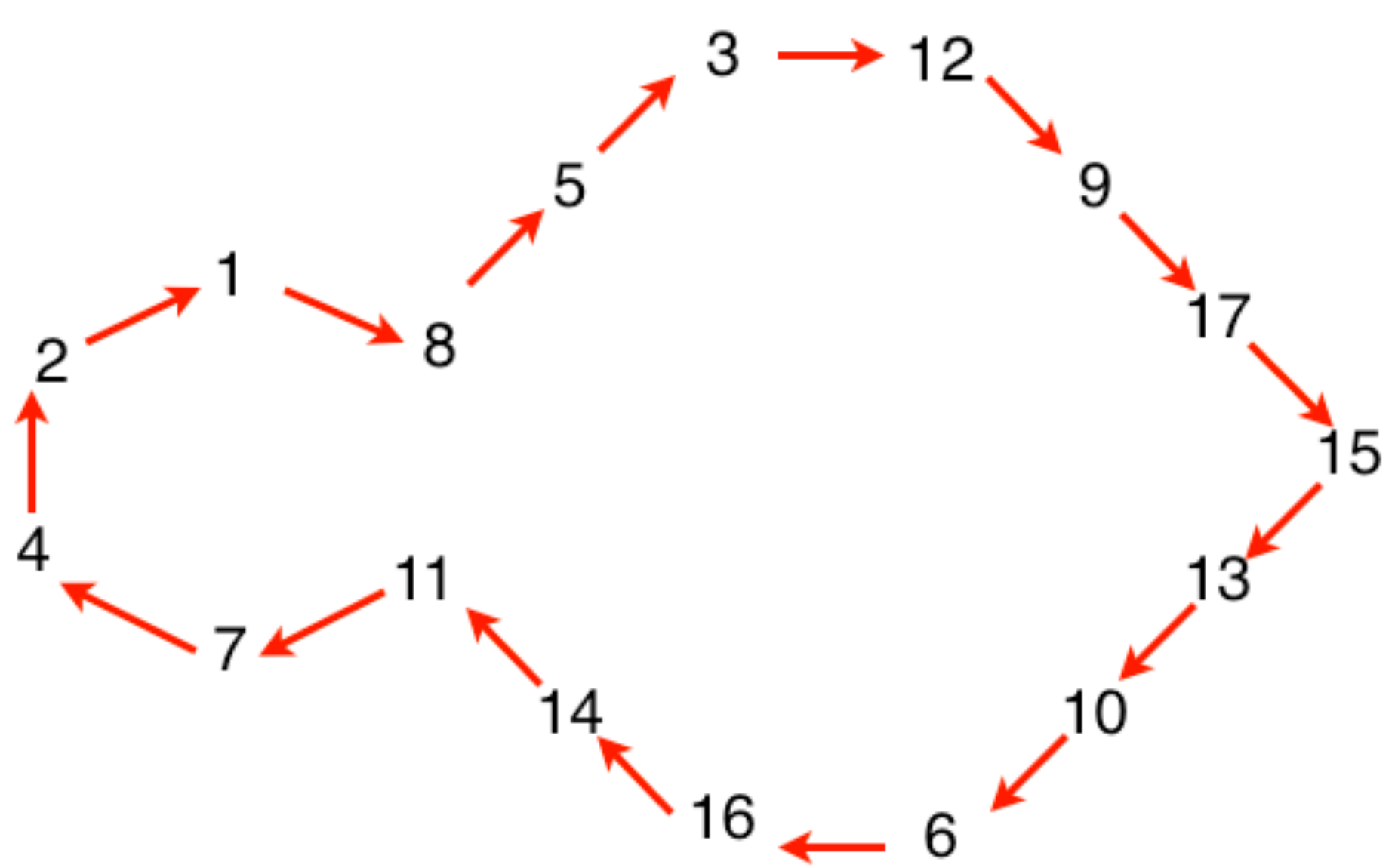}}
$$
\caption{Cycles for $T$ and $T*$.\label{fig:single-double}}
\end{figure}

In fact, starting from the entry $8$ and ending at the entry $11$ of the walk in \eqref{e-1.17}
we get exactly the entries of the cycle $\sigma(T^*)$. This should be so since the algorithm, when applied to the tableau  $T^*$, is not affected by the first column of  $T(D)$.
But when we act on $T(D)$ the edge $11\RA 8$ is replaced by  the $k$-tuple  of edges $11\RA 7\RA 4\RA 2 \RA 1$ augmented by   the edge
$1\RA 8$, resulting in the larger cycle as illustrated in Figure \ref{fig:single-double}.
This completes the inductive proof that $\sigma(D)$ consists of a single cycle.

\section{ More  technical proofs \label{s-proof}}
This section is devoted to give a formal proof of Theorem \ref{t-I.2}.
It is convenient to make the following convention in this section.
We fix a positive integer $k$ and let $m=kn+1$, $n'=n-1$ and $m'=kn'+1$, just as in Lemma \ref{l-1.3}. Let $D\in \CD_{m,n}$ be a Dyck path. Then $D'$ will be in $\CD_{m',n'}$. Let $T=T(D)\in \TAU_{n}^k $ be the tableau
 constructed from $D$ by the filling algorithm
 and $w(T)$ be the closed walk  in the entries of $T$ yielded by  the algorithm of Theorem \ref{t-I.2}.
 Denote by $T^*$ the tableau obtained by removing the first column from $T$. We have seen that applying to $T^*$ the algorithm of Theorem \ref{t-I.2},
as if its letters are contiguous, we obtain a closed walk $w(T^*)$ on the entries of $T^*$. Finally,  it will be convenient to denote by $c_1,c_2,\ldots ,c_{k+1}$  the entries of the first column of $T$.

The closed walk $w(T)$ and $w(T^*)$ are closely related.
\begin{lem}
  \label{l-2.1}
Let $D\in \CD_{kn+1,n}$ be a Dyck path, $T=T(D)$ be its tableau, and
let $T^*$ be obtained from $T$ by removing the first column.
Then the following properties hold true.
\begin{enumerate}
  \item[(1)] As a closed walk $w(T)$ contains the column 1 segment $c_{k+1}\RA c_k\RA \cdots \RA c_1$.

 \item[(2)]  Omitting the column 1 segment from $w(T)$ gives $w(T^*)$. More precisely, if $w(T)$ contains the segment $c'\RA c_{k+1}\RA c_k\RA \cdots \RA c_1\RA c_{k+1}+1$, then replacing this segment by $c'\RA c_{k+1}+1$ gives $w(T^*)$.

\item[(3)] $\l(w(T)=\l(w(T^*))+k+1$.

\item[(4)] $w(T)$ is a closed walk of length $m+n$.

\end{enumerate}
\end{lem}
\begin{proof}

\item {(1)} Since $w(T)$  is a closed walk containing $c_1=1$, it must return to $c_1$. Now $c_i$ has indegree 1 from $c_{i+1}$ for $i=1,2,\dots, k$. It follows that
$w(T)$ must contain the segment $c_{k+1}\to c_k\to \cdots \to c_1$.

\item {(2)} The directed edges of $w(T)$ and $w(T^*)$ are the same if both ends do not involve  column 1 entries. The   directed edges in $w(T)$ that involve column $1$ entries
are $c_{i+1}\to c_i$ for $1\le i\le k$, and $c_1 \to c_{k+1}+1$, together with $c'\to c_{k+1}$ for some entry $c'\in T^*$. We claim that the only  directed edge  in $w(T^*)$ that involves an entry of column $1$ is
 $c' \to c_{k+1}+1$. This is because in $T$ we will go from $c'$ to $c_{k+1}+1$, a bold faced letter, and then to $c_{k+1}$. While in $T^*$, $c_{k+1}+1$ is not bold faced, so in $w(T^*)$  we have the directed edge $c'\to c_{k+1}+1$. This is equivalent to replacing the segment $c'\to c_{k+1} \to c_k\to \cdots c_1 \to c_{k+1}+1$ in $w(T)$ by $c'\to c_{k+1}+1$ to obtain $w(T^*)$.

\item {(3)} This is a direct consequence of   (2).

\item {(4)} Follows from the inductive argument  of last section  and part (3).
\end{proof}

Now we are ready to prove Theorem \ref{t-I.3}, which is restated as follows.
\begin{theo}
  \label{t-2.1}
   On the side of each edge $i\RA j$ of $w(T)$ let us  place an $S$ if $i$ is  in the first row of $T$ and a $W$ otherwise. See Figure \ref{fig:thm-walk-rank} for an illustration.
This done, the  $SW$ sequence of the path  $\Phi^{-1}(D)$ is simply obtained by
 reading all these edge labels starting from $i=1$ and following the directed edges of $w(T)$.
\end{theo}

\begin{figure}
  [!ht]
 $$ \includegraphics[width=3  in]{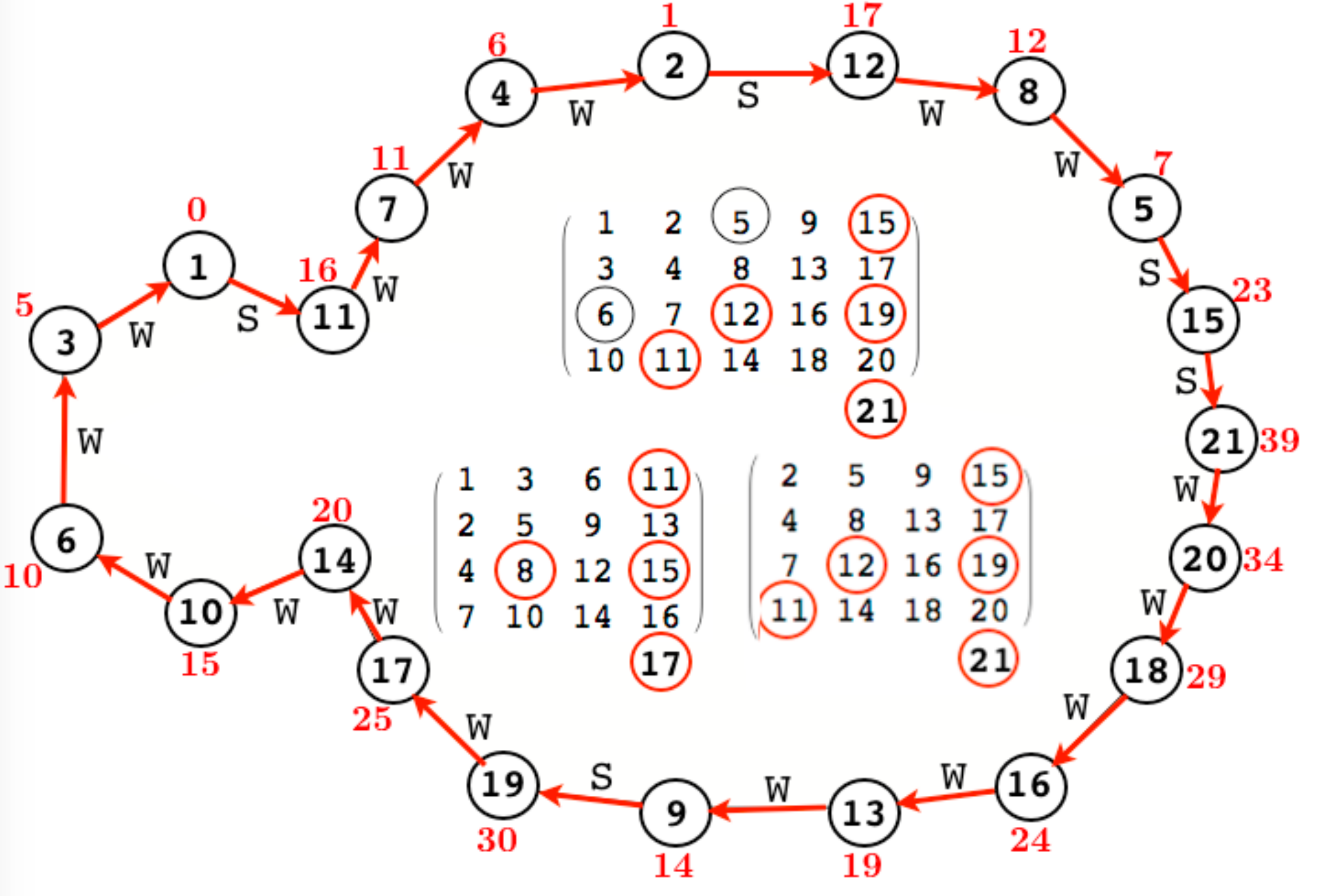}$$
\caption{The outside cycle is $w(T)$ where the tableau $T\in \TAU_{5}^3$ is put on the top. Below are two tableaux: the left one is $T'\in \TAU_4^3$ and the right one is $T^*$. \label{fig:thm-walk-rank}}
\end{figure}

\begin{proof}

Let $\oD$ be the path which results from this $SW$ sequence.
 That is, $\oD$ goes  North at each $S$ and East at each $W$. Since this  $SW$ sequence is a permutation of   $\SW(D)$,  the path $\oD$ will go from $(0,0)$
to $(m,n)$.
Let us compute the sequence of ranks starting by assigning $0$ to
$i=1$ then inductively (following $w(T)$) for each edge $i\RA j$ set
$\rank(j)=\rank(i)+m$ or $\rank(j)=\rank(i)-n$  according as the label of
$i\RA j$ is an $S$ or a $W$, (as illustrated in Figure \ref{fig:thm-walk-rank}).
To show that $\oD$ is a Dyck path we must prove that all these ranks are positive.  We will do this by showing that
\begin{align}
  \label{e-2.2}
  \rank(i+1)-\rank(i)>0  \qquad   \hbox{(for  all $0\le i\le m+n-1$)}.
\end{align}
In fact, this  not only yields that $\oD$ is a Dyck path but we will also     obtain that $\SW(D)$ is a rearrangement  of the steps of $\oD$ by increasing ranks of their starting entries, proving that $\oD=  \Phi^{-1}(D)$.

As we illustrated in Section \ref{s-basic},  we have the identity  $$\rank(j)-\rank(i)=b_{i,j} m-a_{i,j} n$$
 when the path $\Pi_{i\dashrightarrow j}$  from $i$ to $j$
 contains $a_{i,j}$ edge labels $W$  and $b_{i,j}$ edge labels $S$. Then to prove  \eqref{e-2.2} we need only  show that $b_im-a_in>0$ for all $0\le i\le m+n-1$, where $a_i$ and $b_i$ are short for $a_{i,i+1}$ and $b_{i,i+1}$.

We will prove the theorem by induction on $n$. The case $n=1$ is trivial, so we assume the theorem holds for $n'=n-1$. This implies that
if $i<j$ in $T^*$, then
 $$\rank^*(j)-\rank^*(i)=b^*_{i,j} m'-a^*_{i,j} n'>0$$
if $\Pi^*_{i\dashrightarrow j}$ is a path in $w(T^*)$ containing $a^*_{i,j}$ edge labels $W$  and $b^*_{i,j}$ edge labels $S$.

Claim: If $i<j$, and $i,j\ne c_u$ for $u\le k$, then $\rank(j)-\rank(i)>0$.

This is because in the assumed cases, the path $\Pi_{i\dashrightarrow j}$ in $w(T)$ is reduced to a path $\Pi^*_{i'\dashrightarrow j'}$ in $w(T^*)$ by omitting the column 1 segment, where $i'=i+\chi(i=c_{k+1})$ and $j'=j+\chi(j=c_{k+1})$. Clearly  $i'\le j'$.
Now the path
$\Pi_{i\dashrightarrow j}$  contains either all or none of the column 1 segment:
i) if $\Pi_{i\dashrightarrow j}$ does not contain column 1 segment, then the path reduces to $\Pi^*_{i'\dashrightarrow j'}$ with
$b^*_{i',j'}=b_{i,j}$ and $a^*_{i,j}=a_{i,j}$,
and by Lemma \ref{l-1.3} we have
$$\rank^*(j')-\rank^*(i')=b_{i,j}m'-a_{i,j}n'\ge 0 \Rightarrow \rank(j)-\rank(i)=b_{i,j}m-a_{i,j}n\ge 0,$$
where the equality can not hold since $(m,n)$ is a coprime pair;
ii) if $\Pi_{i\dashrightarrow j}$ contains the whole column 1 segment, then
the path reduces to $\Pi^*_{i\dashrightarrow j}$
with $b^*_{i,j}=b_{i,j}-1$ and $a^*_{i,j}=a_{i,j}-k$,
and by Lemma \ref{l-1.3} we have
$$\rank^*(j')-\rank^*(i')=(b_{i,j}-1)m'-(a_{i,j}-k)n'\ge 0 \Rightarrow \rank(j)-\rank(i)=b_{i,j}m-a_{i,j}n\ge 1.$$
This completes the proof of the claim.

Now we prove \eqref{e-2.2} by dealing with three cases (not necessarily mutually exclusive):

\item{Case 1:} $i, i+1 \ne c_u$ for $u\le k$. By the claim, $\rank(i+1)>\rank(i)$.

  \item{Case 2:}  Both $i$ and $i+1$ are in the first column.
  In this case
  we will have  $i=c_u$ and $i+1=c_{u+1} $ for some $1\le u\le k$.
  Since in $w(T)$ we have the
  directed edge $c_{u+1}\RA c_{u}$
  it follows that $\rank(i)-\rank(i+1)=-n$.

\item{Case 3:} If exactly one of $i$ and $i+1$ equals $c_u$ for $u\le k$, then we will transform the path from $i$ to $i+1$ to another path $j$ to $j'>j$ by adding and removing a same number of East steps, so that the Claim applies and we deduce that
$$\rank(i+1)-\rank(i)=\rank(j')-\rank(j)>0.$$
We divide into two subcases as follows.

\item{Case 3a:} if $i=c_u$ and $i+1$ is not in column $1$.
Assume $i+1$ is in another column with entries $d_1,\dots, d_{k+1}$. Since $T$ is row and column increasing
$d_u\ge i+1$, we may assume $i+1=d_v$ for some $v\le u$. Now after $c_u=i$ and $d_v=i+1$ are inserted into $T$, $c_u$ and $d_v$ are both active, the next entries inserted into the two columns must be subsequently $c_{u+1}$ and then $d_{v+1}$, $c_{u+2}$,  and so on. It follows that $c_u<d_v<c_{u+1}<d_{v+1}<\cdots <c_{k+1}<d_{v+k+1-u}$.
Now we transform $\Pi_{i\dashrightarrow i+1}$ to $\Pi_{j\dashrightarrow j'}$, where $j=c_{k+1}$ and $j'=d_{v+k+1-u}>j$ so that the Claim applies. It remains to show that this transform does not change the total weight of the path. Observe that $w(T)$ contains the segment $ d_{v+k+1-u} \to d_{v+k-u} \to \cdots \to d_v$. This is due to the fact that $c_{k+1}+1$ is the smallest bold faced number and that $c_{k+1}>d_{v+k-u}$.
By Lemma \ref{l-2.1} $w(T)$ is a full cycle containing the segment
$c_{k+1}\to c_k\to \cdots \to c_u$.
Thus the path $\Pi_{i}$ looks like
$i=c_u\RA c_{u-1} \RA \cdots \RA d_{v+k+1-u} \to d_{v+k-u} \to \cdots \to d_v=i+1$. By adding $k-u$ East steps at the beginning and removing $k-u$ East steps at the end, we do not change the total weight of the path and obtain $j=c_{k+1}\RA c_{k}\RA \cdots \RA d_{v+k+1-u}=j'$, the path $\Pi_{j\dashrightarrow j'}$ in $w(T)$, as desired.

\item{Case 3b:} if $i+1=c_u$ and $i$ is not in column $1$. The situation is similar to Case 3a. We will include the details here for  convenience.
Assume $i$ is in another column with entries $d_1,\dots, d_{k+1}$. Since $T$ is row and column increasing
$d_u\ge i+2$, we may assume $i=d_v$ for some $v< u$. Now after $d_v=i$ and  $c_u=i+1$ are inserted into $T$, $d_v$ and $c_u$ are both active, the next entries inserted into the two columns must be $d_{v+1}$ and then $c_{u+1}$, $d_{v+2}$,  and so on. It follows that $d_v<c_u<d_{v+1}<c_{u+1}<\cdots<d_{v+k+1-u}<c_{k+1}$.
Now we transform $\Pi_{i\dashrightarrow i+1}$ to $\Pi_{j\dashrightarrow j'}$, where $j=d_{v+k+1-u}$ and $j'=c_{k+1}>j$ so that the Claim applies. It remains to show that this transform does not change the total weight of the path. Observe that $w(T)$ contains the segment $ d_{v+k+1-u} \to d_{v+k-u} \to \cdots \to d_v$. This is due to the fact that $c_{k+1}+1$ is the smallest bold faced number and that $c_{k+1}>d_{v+k+1-u}$.
By Lemma \ref{l-2.1} $w(T)$ is a full cycle containing the segment
$c_{k+1}\to c_k\to \cdots \to c_u$.
Thus the path $\Pi_{i}$ looks like
$i=d_v\RA \cdots \RA c_{k+1} \RA \cdots \RA c_u=i+1$. By adding $k-u$ East steps at the beginning and removing $k-u$ East steps at the end, we do not change the total weight of the path and obtain $j=d_{v+k+1-u}\RA \cdots \RA d_v \RA\cdots \RA c_{k+1}=j'$, the path $\Pi_{j\dashrightarrow j'}$ in $w(T)$, as desired.
\end{proof}

\section{Combinatorial consequences\label{s-consequence}}

It is a simple consequence of the cyclic lemma that for a coprime pair $(m,n)$, the number of Dyck paths in the $m\times n$ rectangle is
\begin{align}
  \label{e-3.1}
{1\over m}{m+n-1\choose n}=\frac{1}{m+n}\binom{m+n}{m}.
\end{align}
In the Fuss case  $m=kn+1$, we have shown that the map $D\RA T(D)$ is a bijection between rational Dyck paths
$\CD_{m,n}$ in the  $m\times n $ lattice rectangle and a class of ``special" $(k+1)\times n$  standard Young tableaux $\TAU_n^k$ which are characterized in Lemma \ref{l-1.4}.

Now given $T\in  \TAU_n^k$ let us denote by $red(T)$ the tableau $T'\in  \TAU_{n-1}^k$ obtained by removing the  first column of $T$ and then reduce the remaining entries to be contiguous integers starting from $1$. The simple algorithm that effects this reduction, is to replace each remaining  letter $i$ by $i-d_i$ if it  is greater than $d_i$ letters in the first column. For instance, the path $D$ below
$$
 \includegraphics[width=5.4   in]{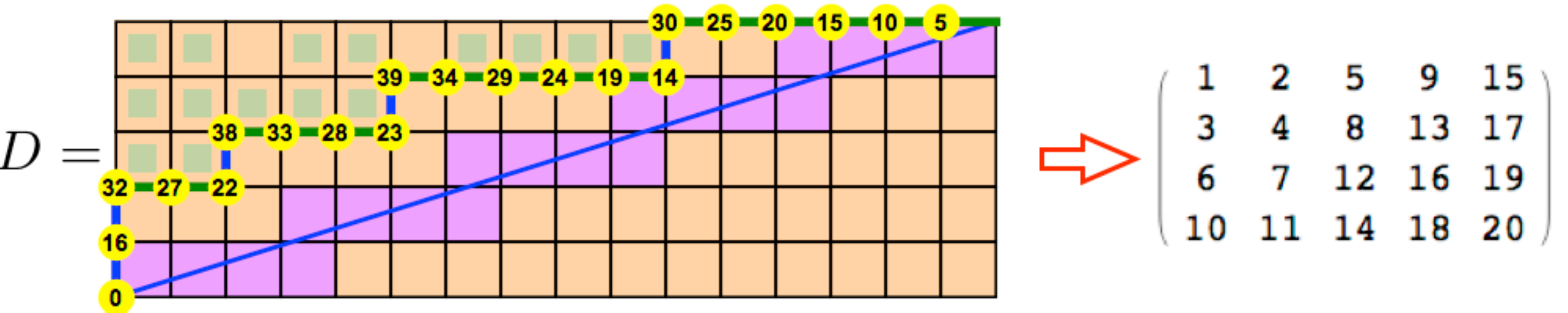}
$$
yields the tableau $T(D)\in \TAU_5^3$
on the  right. In turn, we see below that $T(D)$ reduces to a tableau
$T'\in \TAU_4^3$
\begin{align}
  \label{e-3.3}
red\Big( \begin{pmatrix}
           1 & 2 & 5 & 9 & 15 \\
           3 & 4 & 8 & 13 & 17 \\
           6 & 7 & 12 & 16 & 19 \\
           10 & 11 & 14 & 18 & 20 \\
         \end{pmatrix}
 \Big)
  = \begin{pmatrix}
      1 & 3 & 6 &11 \\
      2 & 5 & 9 & 13 \\
      4 & 8 & 12 & 15 \\
      7 & 10 & 14 & 16 \\
    \end{pmatrix}.
\end{align}

Now it follows from \eqref{e-3.1} and it is geometrically obvious that the
map $D\RA red\big(T(D)\big)$ is necessarily many to one. Thus it is  natural to ask:

\noindent
\textbf{Question:} Given $T'\in \TAU_{n-1}^k$ for how many $D\in\CD_{kn+1,n}$ we have $red\big(T(D)\big)=T'$?

\subsection{Two solutions to the question}

It turns out that not only there is a beautiful answer but there is even a revealing algorithm that constructs all the pre-images of this map. We will give two solutions. The first one relies on the connections between $\Phi^{-1}(D)$
and $\Phi^{-1}(D')$. Our result may be stated as follows.

\begin{theo}[First solution to $red(T)=T'$]
  \label{t-3.1}
  Given $T'\in \TAU_{n-1}^k$,
the number of $D\in \CD_{kn+1,n}$ such that $red\big(T(D)\big)=T'$ is given by the last letter of the first column of $T'$.
Moreover the collection of all solutions to this equation is obtained by the following algorithm:
\begin{enumerate}
\item  Construct the the path $D'\in \CD_{k(n-1)+1,n-1}$ that corresponds to $T'$.

\item  Construct the pre-image $\Phi^{-1}(D')$ and its successive ranks.

\item  Circle all
the ranks of  $\Phi^{-1}(D')$ that are less than $k(n-1)+1$.

\item  Cut $ \Phi^{-1}(D')$ at each one of the circled ranks and reverse the order of the two pieces.

\item Prepend to each of these paths a North step and append $k$ East steps.

\item
The list of the desired $D$ is  obtained by taking the $\Phi$ images
of the resulting Dyck paths.
\end{enumerate}
\end{theo}

To make sure that we understand  the algorithm we will work out the example
\begin{align}
T'=\begin{pmatrix}
     1 & 3 & 6 & 11 \\
     2 & 5 & 9 & 13 \\
     4 & 8 & 12 & 15 \\
     7 & 10 & 14 & 16 \\
   \end{pmatrix}.
\label{e-3.4}
\end{align}

\noindent
In this case the theorem predicts there will be $7$ solutions to the equation
$red\big(T(D)\big)=T'$. Step  $( 1)$ gives
\begin{align}\label{e-D'}
\hspace{2cm} \vcenter{\includegraphics[width=3   in]{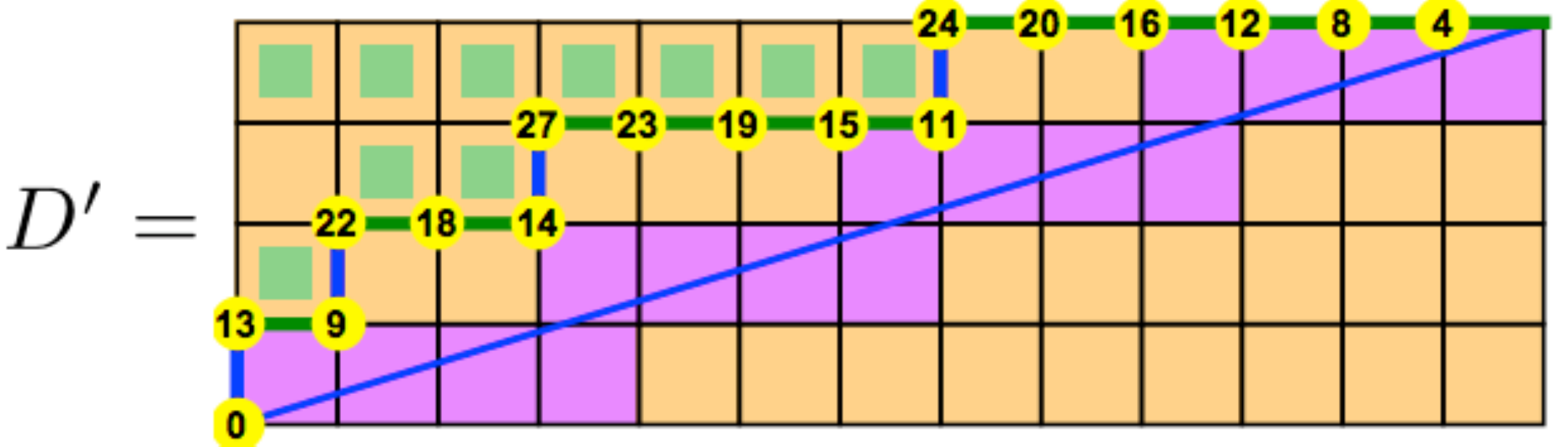}  }
\end{align}

Since in this case $k(n-1)+1=13$, Step
$( 2)$ and Step $( 3)$  give
\begin{align}
  \label{e-bar-D'}
\hskip -.28in\includegraphics[width=3.5   in]{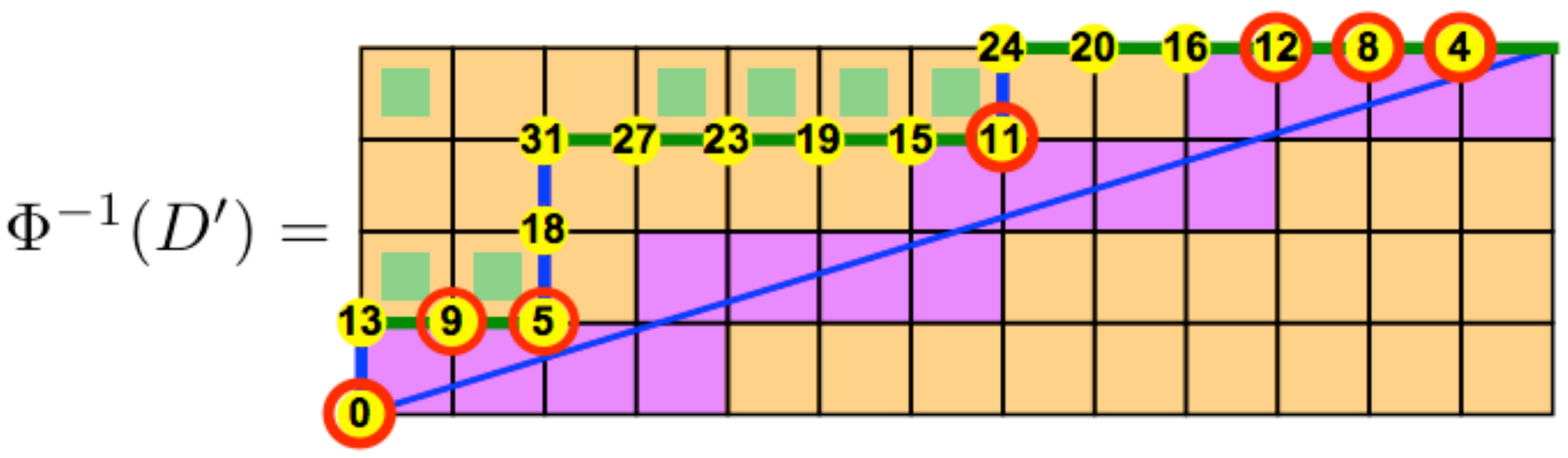}
\end{align}

As predicted we now see $7$ circled ranks. We  carried out below all but the last step  only for  $6$ of the circled ranks. Namely, ranks $9,5,11,12,8,4$. Figure \ref{fig:cutsix} below exhibits the ``cuts'', the
``reorders'' and the construction of the
corresponding $\Phi^{-1}(D)$. The rank of the node of the cut is placed under each path.
\begin{figure}
  [!ht]
$$
\hskip -.35 in\includegraphics[width=3.5   in]{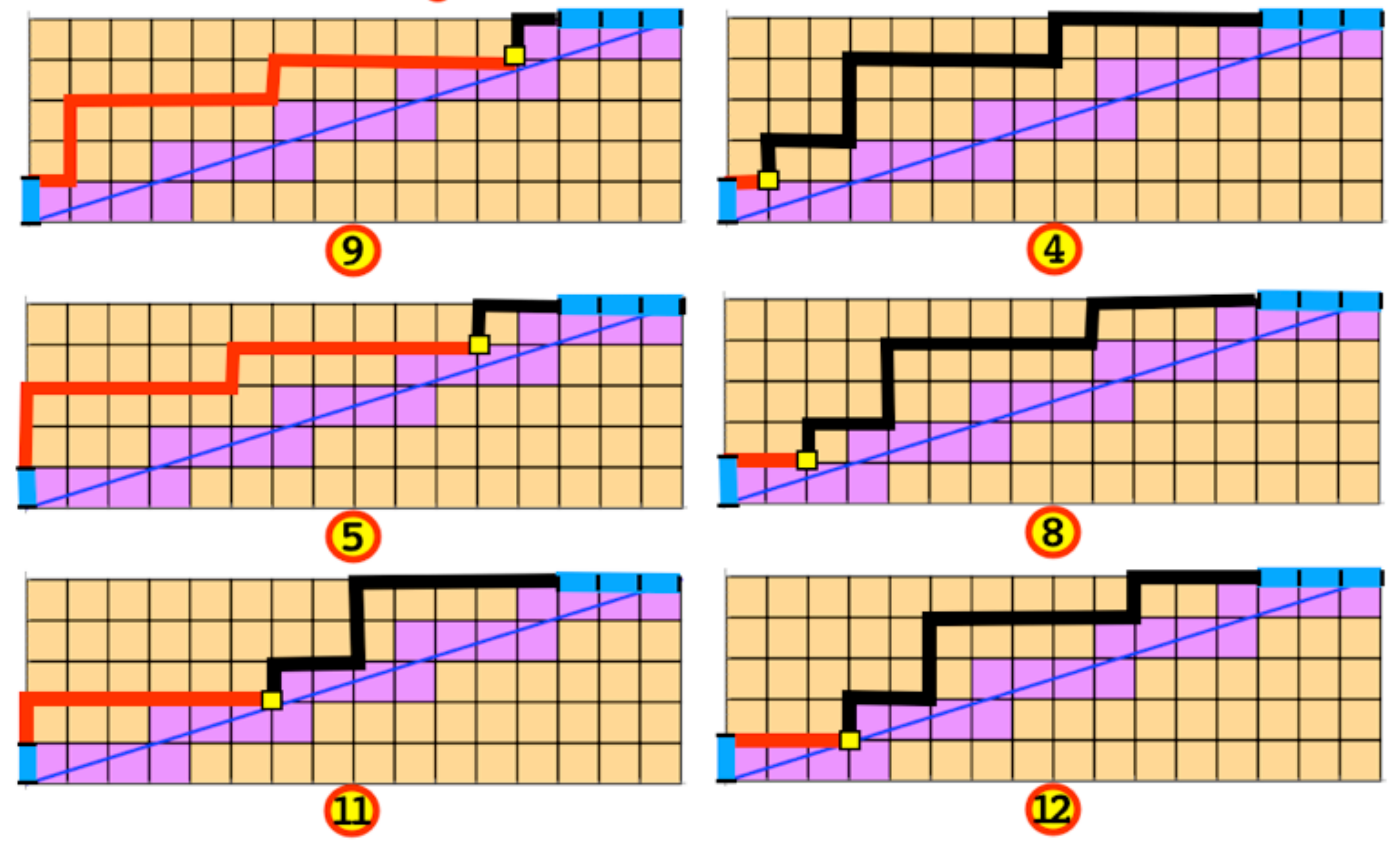}
$$
\caption{$6$ solutions correspond to $red(T)=T'$.\label{fig:cutsix}}
\end{figure}

Let us explain first this geometric view of the map $T\to red(T)$.
\begin{lem}\label{l-D-red}
  Let $T=T(D)$ for a Dyck path $D\in \CD_{kn+1,n}$. Then there is a unique
   Dyck path $D'\in \CD_{kn'+1,n'}$, where $n'=n-1$ such that $red(T(D))=T(D')$. Moreover, $\Phi^{-1}(D')$ is obtained from $\Phi^{-1}(D)$ by the following process: i) remove the starting North step and the ending $k$ East steps; ii) split the resulting path as $BA$ at the node with the smallest rank; iii) set $\Phi^{-1}(D')=AB$, the circular rearrangement of $BA$.
\end{lem}
\begin{proof}
Recall that the SW sequence of $\Phi^{-1}(D)$ is obtained from the closed path $\Pi_{1\dashrightarrow 1}$ from $1$ to $1$ in the closed walk $w(T)$ by simply replacing each $i\RA j$ by an $S$ if $i<j$ (or equivalently, $i$ is in row 1 of $T$)
 and by a $W$ if $i>j$ (or equivalently, $i$ is not in row 1 of $T$).

Now notice that
$$\Pi_{1\dashrightarrow 1}=1\RA v\RA \cdots \RA (\uparrow 1')\RA \cdots \RA u\RA c_{k+1}\cdots \RA c_2\RA 1,  $$
where we used $(\uparrow 1')$ to denote the updated label of $1'$ in $T$. Thus the SW sequence of $\Phi^{-1}(D)$ starts with an $S$ and ends with $k$ copies of $W$'s. The middle part from $v$ to $u$ corresponds to $BA$
where $A$ is obtained from the path $\Pi_{1'\dashrightarrow u'}$ in $w(T')$ and $B$ is obtained from $\Pi_{v'\dashrightarrow 1'}$.
This is equivalent to saying that if we cut $\Phi^{-1}(D')$ at certain starting rank as $AB$, i.e., $A$ followed by $B$, then $\Phi^{-1}(D)=SBAW^k$
must contain the cyclic rearrangement $BA$ of the steps of $D'$.

Finally, since $B$ ends with rank $0$ in $D'$
and $A$ starts with rank $0$ and all their other ranks (in $D'$) are positive, we see that to separate the middle part of $\Phi^{-1}(D)$ as $BA$, we have to cut at the smallest rank.
\end{proof}

\begin{proof}
  [Proof of Theorem \ref{t-3.1}]
Given $T'=T(D')$ we can construct the closed walk $w(T')$ of $T'$ and hence the Dyck path $\Phi^{-1}(D')$.
To construct $T=T(D)$ with $red(T)=T'$, it is sufficient to construct $\Phi^{-1}(D)$.

By Lemma \ref{l-D-red}, such $\Phi^{-1}(D)$ must be obtained by: i) First cut $\Phi^{-1}(D')$ at a node as $AB$; ii) construct the circular rearrangement $BA$ and then $SBAW^k$.
 We claim that the only circular rearrangements of $D'$ that can contribute
to such a construction of $\Phi^{-1}(D)$ are those obtained by cutting $D'$ at a node of rank smaller than $m'=k(n-1)+1$.

To see this assume we cut $\Phi^{-1}(D')$ at a node of rank $r_i$ to have $\Phi^{-1}(D')=AB$. We need to check for which $r_i$, the constructed path $\WD= SBAW^k$ is in $\CD_{m,n}$. Here we are involved in using two rank systems, one for $\CD_{m',n'}$ and the other for $\CD_{m,n}$.
Lemma \ref{l-1.3} says that positive ranks for $\CD_{m',n'}$ are still positive when treated in $\CD_{m,n}$. Thus changing the rank system from $\CD_{m',n'}$ to $\CD_{m,n}$ does not change the relative order of the ranks. (Note that rank $0=n'm'-m'n'$ becomes $-1=n'm-m'n$, so we need the precise range as stated in Lemma \ref{l-1.3}.)

If we use the rank system of $\CD_{m',n'}$, i.e., starting with $0$ at the beginning, adding $m'=k(n-1)+1$ after a North step, and subtracting $n-1$ after an East step, then the ranks of $\WD$ are: i) the ranks of $AB$ shifted by adding $m'-r_i$ so that $BA$ starts at rank $m'$;
ii) $0$ as the rank of the starting $S$ and $m'-j(n-1)$ where $0\le j\le k-1$ as the ranks for the ending $k$ $W$'s. Thus the smallest rank is $\min(0,m'-r_i)$, which is $\ge 0$ only if $m'>r_i$, and by Lemma \ref{l-1.3} the minimal rank of $\WD$ is also $0$ when treated as in $\CD_{m,n}$. This implies cutting at a node of rank $r_i<m'$ do contribute. For those $r_i>m'$,
 $m'-r_i$ remains negative in the rank system of $\CD_{m,n}$ and hence cutting at a node of rank $r_i>m'$ does not contribute. For $r_i=m'$, we have
$A=S$ so that $\WD=SBSW^k$, which is clearly not in $\CD_{m,n}$ since the rank at the end of $B$ is $-1$.

Finally, for every such candidate, we remove the first North step and the final $k$ East steps, and split the resulting path at the smallest rank as $BA$. Then $AB$ will be the desired Dyck path $\Phi^{-1}(D')$ preserving all the relative order of the ranks.

This not only proves our claim, but also completes our proof of Theorem \ref{t-3.1} since it also beautifully explains the fact that the bottom entry of the first column of $T(D')$ predicts the number of solutions of the equation $red\big(T(D)\big)=T(D')$.
In fact, the algorithm for constructing all the solutions of this equation is also an immediate consequence of all the
observations we made during our proof.

\begin{figure}[!ht]
$$
\includegraphics[width=4   in]{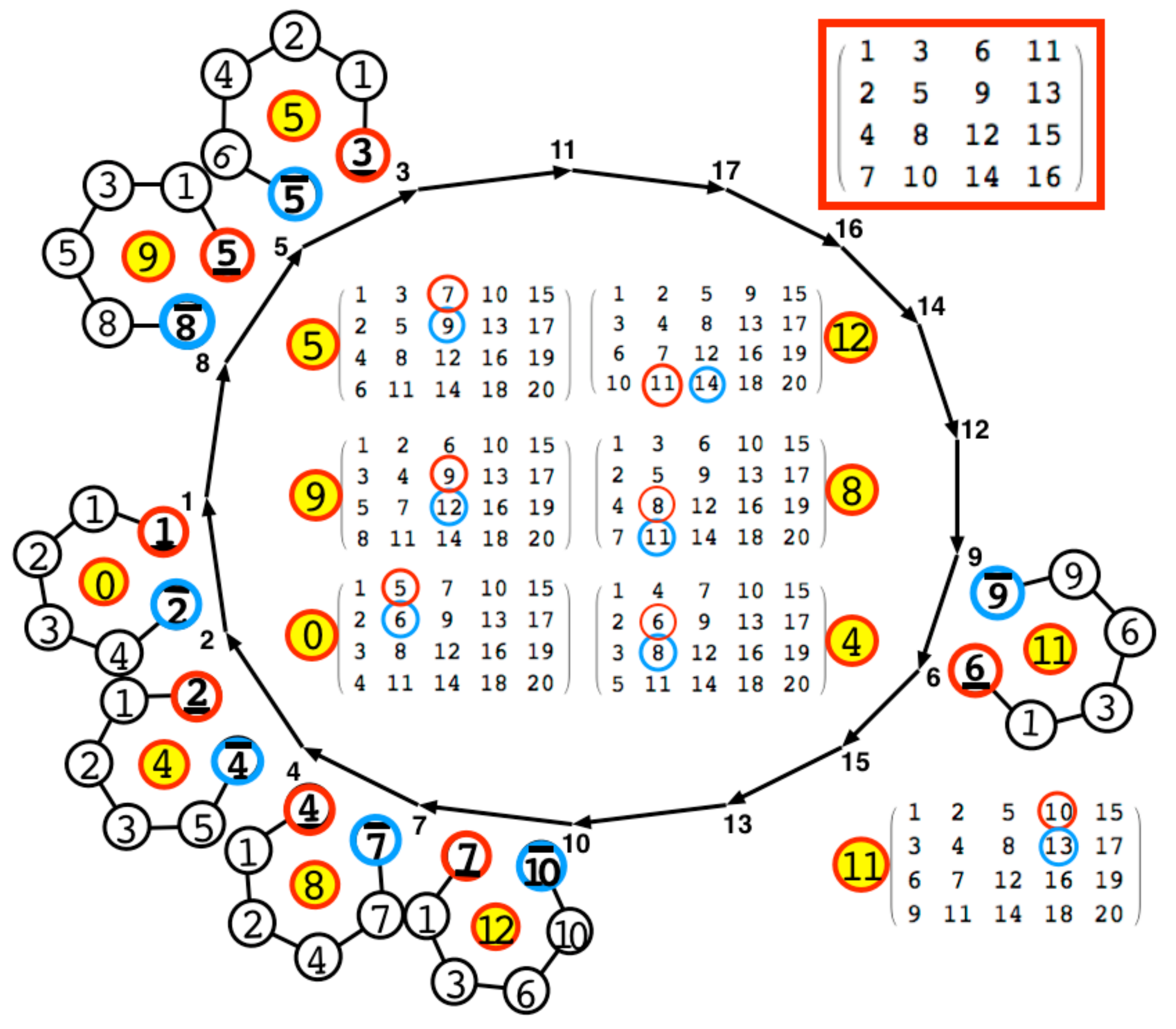}.
$$
\caption{A big example constructing $w(T)$ from $w(T')$ where $red(T)=T'$.\label{fig:chaos}}
\end{figure}

 All these observations should be quite evident in Figure \ref{fig:chaos}. There the tableau $T'$ is highlighted by a red frame and put to the right upper corner.
Each of the loops enters $w(T')$ at the (overlined)\footnote{we distinguish by overline or underline for black-white print.} node surrounded by a blue circle and exits
$w(D')$ at a red  circled (underlined) node. At the
center of each loop we inserted the rank at which $\Phi^{-1}(D')$ is cut. Likewise, each tableau $T(D)$ of the solution $D$ of the equation $red\big(T(D)  \big)=T(D')$ is  labeled by the rank at which $\Phi^{-1}(D')$ is cut.
The list of the tableaux $T(D)$ is ordered as closely as possible to the corresponding loops. In each tableau $T(D)$ we circled the labels that correspond  to the entry
from the loop to $w(D')$  and exit from
$w(D')$ to the loop as a walk through the entries in $T(D)$.
\end{proof}

In Figure \ref{fig:chaos}, if we focus on the bottom row of the $T$'s, we will see that they only differ at the first entry.
This leads to our second solution to finding $T$ with $red(T)=T'$.

Our result relies on more characterizations of $T\in \TAU_n^k$. Firstly, the characterization property of $T\in \TAU_n^k$ in Lemma \ref{l-Tnk} also leads to a natural involution $\psi: \TAU_{n}^{k} \mapsto \TAU_{n}^{k}$, by first flipping vertically, then flipping horizontally, and finally reversing the the entries. More precisely, the $(i,j)$-entry of $\psi(T)$ is given by $\psi(T)_{i,j}=(k+1)n+1-T_{k+2-i,n+1-j}$. Clearly we have $\psi(T)\in \TAU_{n^k}$ since it satisfies the row increasing, column increasing, and the $a<b<c<d$ condition.

The following characterization follows directly from Lemma \ref{l-1.4}, Corollary \ref{c-S-positions} and the involution $\psi$.
\begin{lem}\label{l-Tnk}
A tableau $T\in \mathcal T_{n}^k $ is uniquely determined by its first row entries. It is also uniquely determined by its bottom row entries. Conversely, we have
 \begin{enumerate}
 \item[i)]
    an increasing sequence $(t_1,t_2,\dots, t_n)$ with $t_1=1$ is the top row entries of $T\in \mathcal T_{n}^k $ if and only if
$t_j \le 1+(j-1)(k+1)$ for all $j$;
\item[ii)]
an increasing sequence $(b_1,b_2,\dots, b_n)$ with $b_n=(k+1)n$ is the bottom row entries of $T\in \mathcal T_{n}^k $ if and only if
$b_j \ge j(k+1)$ for all $j$.

 \end{enumerate}
\end{lem}

Note that given the bottom row entries of $T$, we can construct the whole tableau $T$ as follows: i) construct the top tow entries of $\psi(T)$;
ii) use the filling algorithm to construct the whole tableau $\psi(T)$;
iii) obtain $T$ by applying $\psi$ to $\psi(T)$.

Now we are ready to present and prove our second solution.
\begin{theo}[Second solution to $red(T)=T'$]
 Given $T'\in \TAU_{n-1}^k$,
the number of $T\in \TAU_n^k$ such that $red(T)=T'$ is given by the last letter of the first column of $T'$.
Moreover, if the bottom rows of $T$ and $T'$ are $b_1,\dots, b_n$ and  $b_1',\dots, b_{n-1}'$ respectively, then
$b_j= b_{j-1}'+k+1$ for $2\le j\le n$, and $b_1$ can be any one of the $b_1'$ numbers $k+1,k+2,\dots, k+b_1'$.
\end{theo}
\begin{proof}
The equality $b_j=b_{j-1}'+k+1$ is clear, since the column $1$ entries of $T$ are all less than $b_j$ for $j\ge 2$.

The other parts follow directly from the unique characterization in Lemma \ref{l-Tnk} of $T\in \TAU_{n}^k$ by the $b_j$'s. The only remaining condition is $k+1\le b_1<b_2$, as desired.
\end{proof}

\subsection{The Higher $q,t$-Catalan polynomials}
The $q, t$-Catalan polynomials were first introduced by
Garsia and Haiman \cite{qt-Catalan} in 1996.
The identity
 \begin{align}
   C_n(q,t) = \sum_{D\in \CD_{n,n}} q^{\dinv(D)} t^{\area(D)}=\sum_{D\in \CD_{m,n}} q^{\area(D)} t^{\bounce(D)},
 \end{align}
was proved in \cite{qt-Catalan-pos}, where we simply set $\bounce(D)=\area(\Phi^{-1}(D))$. It is referred to as the classical case, and plays a prominent role in combinatorics, symmetric function
theory, and algebraic geometry.

The higher $q,t$-Catalan polynomials $C_n^{(k)}(q,t)$, also introduced in the same paper \cite{qt-Catalan}, are natural generalizations of the classical case. We will use the following combinatorial form due to \cite{Loehr-higher-qtCatalan}:
 \begin{align}
   C_n^{(k)}(q,t) = \sum_{D\in \CD_{kn+1,n}} q^{\area(D)} t^{\bounce(D)}=\sum_{D\in \CD_{m,n}} q^{\dinv(D)} t^{\area(D)}.
 \end{align}
Again we simply set $\bounce(D)=\area(\Phi^{-1}(D))$.
It is known that the sweep map $\Phi$ takes $(\dinv,\area)$ to $(\area,\bounce)$, which generalizes the $\zeta$ map for the classical case in \cite{Hag-book08}. However, direct combinatorial interpretation of the bounce statistic is only known for the Fuss case. See \cite{Loehr-higher-qtCatalan}.

Our purpose in this subsection is to establish the following formula for higher Catalan polynomials.
\begin{prop}
Let $m'=kn'+1$ where $n'=n-1$. Then we have
 \begin{align}\label{e-area-bounce}
   C^{(k)}_n(q,t) = \sum_{D'\in \CD_{m',n'}} q^{\area(D')} t^{bounce(D')} \sum_{i=1}^{m'+n'} \chi(\bar r_i(D')<m') q^{i-1} t^{n'k-\bar r_i(D')},
 \end{align}
 where $\bar r_i(D')$ is $i$-th smallest rank in the rank sequence of $\Phi^{-1}(D')$. Or equivalently,
 \begin{align}\label{e-dinv-area}
   C^{(k)}_n(q,t) = \sum_{D'\in \CD_{m',n'}} q^{\dinv(D')} t^{\area(D')} \sum_{i=1}^{m'+n'} \chi(r_i(D')<m') q^{i-1} t^{n'k- r_i(D')},
 \end{align}
 where $r_i(D')$ is the $i$-th smallest rank in the rank sequence of $(D')$.
\end{prop}
Equation \eqref{e-dinv-area} is simply obtained by applying the sweep map $\Phi$ to \eqref{e-area-bounce}, which itself is suggested by the two solutions in the previous subsection. These solutions give close connections between
$T'$ and those $T$ with $red(T)=T'$. Indeed, the only thing we need to show is Corollary \ref{c-area} and Theorem \ref{t-bounce} below.

To this end, we need to establish some formulas about the area statistic. Recall that
$\area(D)$ is the the number of cells between $D$ and the diagonal. It is easy to see that the maximal area is $k\binom{n}{2}$, so we will also use
$\coarea(D)=k\binom{n}{2}-\area(D)$ for the number of cells above $D$.
Similarly, we use $\cobounce(D)=k\binom{n}{2}-\bounce(D) $.

\begin{prop}\label{p-area-T(D)}
In the Fuss case $m=kn+1$,   if $T=T(D)$ for a Dyck path $D\in \CD_{m,n}$ has top row entries $(t_1,\dots, t_n)$ and bottom row entries $(b_1,\dots, b_n)$, then we have
\begin{align}
  \coarea(D) &= \sum_{i=1}^n t_i - \binom{n+1}{2}, \label{e-coarea-t}\\
  \area(D) &= \sum_{i=1}^n b_i -(k+1) \binom{n+1}{2} \label{e-area-b}.
\end{align}
\end{prop}
\begin{proof}
By Lemma \ref{l-Tnk}, the increasing sequence $(t_1,\dots, t_n)$ uniquely determines $D$, and they give the positions of the $S$'s of $D$. It is clear that $t_i-i$ is just the number of $W$'s before
the $i$-th $S$ in $D$, which is also the number of lattice squares in row $i$ to the left of $D$. It then follows that
$$ \coarea D = \sum_{i=1}^n (t_i-i)= \sum_{i=1}^n t_i - \binom{n+1}{2}.$$
This is just \eqref{e-coarea-t}.

We will use the first part to prove \eqref{e-area-b}. By
apply Proposition \ref{p-rank-complement} to our $\oD=\Phi^{-1}(D)$, we obtain
\begin{align*}
 \coarea(D)= \coarea(\Phi(\widehat{\oD})) &= \sum_{i=1}^n (m+n-b_i) -\binom{n+1}{2},
\end{align*}
since the positions of the $N$'s in $\NE(D)$ are exactly $b_i+1$ for $1\le i\le n$.
It follows that
\begin{align*}
  \area(D) &=k\binom{n}{2}-  \sum_{i=1}^n (m+n-b_i) +\binom{n+1}{2}\\
  &=\sum_{i=1}^n b_i + \frac12 kn(n-1) - n(kn+n+1) +\frac{1}{2}n(n+1)\\
  &=\sum_{i=1}^n b_i -(k+1)\binom{n+1}{2}.
\end{align*}
This completes the proof.
\end{proof}

A direct consequence is the following corollary.
\begin{cor}\label{c-area}
In the Fuss case $m=kn+1$, suppose $T=T(D)$ reduces to $T'=T(D')$
for $D\in \CD_{m,n}$.   If $b_1(T)=k+j$ (i.e., the first entry of the bottom row of $T$), then
$$ \area(D)=\area(D')+j-1.$$
\end{cor}

The next result needs some work.
\begin{theo}\label{t-bounce}
Suppose in the Fuss case $m'=kn'+1$, we cut $\Phi^{-1}(D')\in \CD_{m',n'}$ at a node of rank $r<m'$ as $AB$, and set $\Phi^{-1}(D)=SBAW^k$. Then
$$ \cobounce (D) = \cobounce (D') + r. $$
Equivalently,
$$ \bounce(D)=\bounce(D')+n'k-r.$$
\end{theo}
\begin{proof}
Let $m,n,m',n'$ as before.
Let $T=T(D)$, $T^*$ be obtained from $T$ by removing column 1 entries, and let $T'=T(D')$. Then $T'$ is obtained from $T^*$ by reducing the entries to be contiguous with smallest entry $1$.
Assume column 1 entries of $T$ are $c_1,\dots, c_{k+1}$. Then $c_1=1=t_1$ and $c_{k+1}=b_1$. The closed walk of $T$ is
$$w(T)=1\to b_1+1\to \cdots \cdots \to b_1 \to c_k\to \cdots \to c_2 \RA 1,$$
and $w(T^*)$ is simply obtained from $w(T)$ by omitting column 1 elements of $T$. However to construct $\Phi^{-1}(D')$, we need to start at $t_2$ in $w(T^*)$.

Now $r<m'$ can be uniquely written as $\beta m'-\alpha n'$. This means that in $w(T^*)$, the path $\Pi^*_{t_2\dashrightarrow b_1+1}$ from $t_2$ to $b_1+1$ has $\beta $ letters $S$ and $\alpha $ letters $W$. Then there are $\delta =m'+n'+1-a-b$ edges along the path $\Pi_{1\dashrightarrow t_2}$, and after the node $t_2$, there are exactly $\beta$ entries coming from row $1$ of $T$.

More precisely, if the positions of row 1 entries of $T$ in $w(T)$ are $u_1< u_2<\cdots< u_n$, then $t_1$ is at the position $u_1=1$ and $t_2$ is at the position $u_{n-\beta+1}=\delta +1$. It follows that the positions of row 1 entries of $T$ (excluding $t_1=1$) in $w(T^*)$ is given by
$$ \{ u_{n-\beta+1}, \dots, u_n, u_2+m'+n', \dots, u_{n-\beta}+m'+n'\} -\delta .$$
That is to say, $t_2$ is reset to position $1$, so that positions $u_j$ for $n-\beta+1\le j \le n$ becomes $u_j-\delta$, and positions $u_j$ for $2\le j \le n-\beta$ becomes $u_j-\delta+m'+n'$ due to the cyclic rearrangement. It then follows that
\begin{align*}
  \cobounce (D') &= \sum_{i=2}^n u_i -(n-1) \delta +(n-\beta-1)(m'+n') -\binom{n}{2} \\
                 &=\sum_{i=1}^n u_i -\binom{n+1}{2} +n-1 -\beta(m'+n')+(n-1)(m'+n'-\delta) \\
                 &= \cobounce (D) + n' -\beta(m'+n') +n'(\alpha+\beta-1)\\
                 &= \cobounce (D)  -\beta m' + \alpha n' =\cobounce(D)-r.
\end{align*}
This completes the proof.
\end{proof}
We illustrate by an example.
In Figure \ref{fig:chaos}, using the bottom row of $T'\in \TAU_4^3$ gives
$$ \area(D')= 7+10+14+16 -(3+1)\binom{4+1}{2} =7,$$
which agree with the direct count for $D'$ in \eqref{e-D'}. Now look at the $T$ corresponding to rank $9$, which is the first tableau in the second row. We have
\begin{align}
  \area(D)=8+11+14+18+20-(3+1)\binom{5+1}{2}=11=\area(D')+4.
\end{align}
This agrees with Corollary \ref{c-area}, since $8-(3+1)=4$.

For cobounce of $D'$ and $D$, we need to look at the picture of $\Phi^{-1}(D')$ in \eqref{e-bar-D'} and $\Phi^{-1}(D)$ in Figure \ref{fig:cutsix}, the first picture. Direct count gives
$\cobounce(D')=11$ and $\cobounce(D)=20$. This agrees with Theorem \ref{t-bounce}.

\section{Concluding Remark}
In this paper we have presented an $O(m+n)$ algorithm for inverting the sweep map in the Fuss case $m=kn\pm 1$ by introducing an intermediate object $\TAU_{n}^k$. The inverse bijection $\Phi^{-1}$ is then decomposed into two easy steps, by first constructing $T=T(D)$ from the $SW$ sequence of $D$ by the filling algorithm, and then produce $\Phi^{-1}(D)$ by constructing a closed walk on $T$. The proof is lengthy, but this is the usual situation: the easier the algorithm is, the harder to prove its bijectivity.

Our algorithm for the Fuss case raises a natural question: Is there an $O(m+n)$ algorithm to invert the sweep map for general $m$ and $n$. Our hope is to find an intermediate object replacing $\TAU_n^k$, but so far we have not succeeded.

It will be interesting to find a direct combinatorial interpretation of the $\bounce$ statistic for general $m$ and $n$.

Identity \eqref{e-dinv-area} can be used to give a recursive algorithm for computing higher $q,t$-Catalan polynomials, however in such a formula we need to keep track of the ranks of $D'$ that are smaller than $m'$. The formula seems too complicated to be included here.

{\small \textbf{Acknowledgements:}

%
The first named author was supported by NFS grant DMS13--62160.

\end{document}